\vspace{1cm}%\documentclass[11pt]{article}
\documentclass[numbook, envcountsect, envcountsame, envcountreset, runningheads, smallextended]{article}
\usepackage{url}
\usepackage{amsmath}
\usepackage{amssymb}
\usepackage{amsthm}
\usepackage{empheq}
\usepackage{latexsym}
\usepackage{dsfont}
\usepackage{appendix}
\usepackage{color} 
\usepackage[utf8]{inputenc}
\usepackage[T1]{fontenc}
\usepackage{geometry}
\usepackage{breqn}
\usepackage{setspace}

\usepackage{eurosym}
\usepackage{authblk}

\usepackage{graphicx}
\usepackage{caption}
\usepackage{subcaption}

\usepackage{titlesec}

\titleformat{\section}
    {\normalfont\Large\bfseries\scshape}{\thesection }{.5em}{\vspace{.5ex}}[\titlerule]
\titleformat*{\subsection}{\large\scshape\bfseries}

\newtheorem{Theorem}{Theorem}[part]

\newtheorem{Lemma}{Lemma}[part]
\newtheorem{Corollary}{Corollary}[part]
\newtheorem{Remark}{Remark}[part]

\makeatletter \@addtoreset{equation}{section}

\@addtoreset{Definition}{section}

\@addtoreset{Theorem}{section}

\@addtoreset{Proposition}{section}

\@addtoreset{Property}{section}

\@addtoreset{Assumption}{section}

\@addtoreset{Corollary}{section}

\@addtoreset{Lemma}{section}

\@addtoreset{Remark}{section}

\@addtoreset{Example}{section}

\@addtoreset{Condition}{section}

\def \E{\mathbb{E}}

\def \P{\mathbb{P}}

\def \R{\mathbb{R}}

\def\Ac{{\cal A}}
\def\Bc{{\cal B}}

\def\Fc{{\cal F}}

\def\Wc{{\cal W}}

\def \ind{1\!\!1}

\begin{document}

\title{A Class of Explicit optimal contracts in the face of shutdown}

\author[1]{Jessica Martin\footnote{jessica.martin@insa-toulouse.fr}}

\author[2]{Stéphane Villeneuve\footnote{stephane.villeneuve@tse-fr.eu\\ The authors acknowledge funding from ANR PACMAN and from ANR under grant ANR-17-EUR-0010 (Investissements d'Avenir program). The second author gratefully thanks the  FdR-SCOR ``Chaire March\'e des risques et cr\'eation de valeurs".}}

\affil[1]{INSA de Toulouse\\ IMT UMR CNRS 5219\\ Universit\'e de Toulouse\\ 135 Avenue de Rangueil 31077 Toulouse Cedex 4 France}

\affil[2]{Université Toulouse 1 Capitole, (TSE-TSMR)\\ 2 Rue du Doyen Gabriel Marty\\ 31000 Toulouse}

\date{\today}
\maketitle
\begin{abstract}
What type of delegation contract should be offered when facing a  risk of the magnitude of the pandemic we are currently experiencing and how does the likelihood of an exogenous early termination of the relationship modify the terms of a full-commitment contract? We study these questions by considering a dynamic principal-agent model that naturally  extends the classical Holmström-Milgrom setting to include a  risk of default whose origin is independent of the inherent agency problem. We obtain an explicit characterization of the optimal wage along with the optimal action provided by the agent. The optimal contract is linear by offering both a fixed share of the output which is similar to the standard shutdown-free Holmström-Milgrom model and a linear prevention mechanism that is proportional to the random lifetime of the contract.  We then tweak the model to add a possibility for risk mitigation through investment and study its optimality. 
\end{abstract}


\section{Introduction}

Without seeking to oppose public health and economic growth, there is no doubt that the management of the Covid crisis had serious consequences on entire sectors of the economy. The first few months of 2020 will go down in world history as a period of time characterized by massive layoffs, forced closures of non-essential companies, disruption of cross-border transportation whilst populations were subject to lockdown and/or social distancing measures and  hospitals and the medical world struggled to get a grasp on the Sars-Cov-2 pandemic. Whilst the immediate priority was saving lives, decongesting hospitals and preventing the spread of the disease, many extraordinary economic support measures were taken to help businesses and individuals stay afloat during these unprecedented times and in the hope of tempering the economic crisis that would follow. Although the world has lived through many crises over the past centuries, from several Panics in the 1800s and the Great Depression of the 1930s to the more recent Financial Crisis of 2008, never before has the global economy as a whole come to such a standstill due to an external event.  Such large shutdown risks do not only materialize during pandemics but throughout many major other large events. The massive bushfires that affected Australia towards the end of 2019, temporarily halting agriculture, construction activity and tourism in some areas of the country are another recent example. 
As we begin to see a  glimpse of hope for a way out through vaccination, the focus is turning to building the world of tomorrow with the idea that we must learn to live with such risks.  This paper tries to make its contribution by focusing on a simple microeconomic issue. In a world subject to moral hazard, how can we agree to an incentive contract whose obligations could be made impossible or at least very difficult because of the occurrence of a risk of the nature of the Covid19 pandemic? Including such a shutdown risk-sharing in contracts seems crucial going forward for at least two reasons. First, it is not certain that public authority will be able to continue to take significant economic support measures to insure the partners of a contract if the frequency of such global risks were to increase. On the other hand, the private insurance market does not offer protection against the risk of a pandemic which makes pooling too difficult. It therefore seems likely that we will have to turn to an organized form of risk sharing between the contractors.\\

Economic theory has a well-developed set of tools to analyze incentive and risk-sharing problems using expected-utility theory. Most of the now abundant literature related to dynamic contracting through a Principal-Agent model has, so far, mostly been based on continuously governed (eg. Brownian motion) output-processes. This was the case of the foundational work of Holmström and Milgrom \cite{HM} and many of its many extensions such as those of Schattler and Sung in \cite{SS} and \cite{SSung97}, and the more recent contributions of Sannikov in \cite{Sannikov08}, and Cvitanic et al \cite{CXZ} and \cite{CPT:18}. However some relatively recent works have introduced jump processes into continuous time contracting. Biais et al. were the first to do so in \cite{Biais} where they study optimal contracting between an insurance company and a manager whose effort can reduce the occurence of an underlying accident.  In a similar vein, work by Capponi and Frei \cite{CF} also used a jump-diffusion driven outcome process in order to include the possibility for accidents to negatively affect revenue. 
Here, we extend the classical Holmström and Milgrom \cite{HM} framework to include a shutdown risk.  We do not claim that this model with CARA preferences is general  enough to come up with robust economic facts, but it has the remarkable advantage of being explicitly creditworthy, which allows us to find an explicit optimal contract that disentangles the incentives  from external risk-sharing and allows us to understand the sensitivity of the optimal contract to the different exogenous parameters of the model. Our work uses a jump-diffusion process too and presents some structural similarities with \cite{CF} : both models consider a risk averse principal and agent with exponential utility and reach an explicit characterization of the optimal wage. However, Capponi and Frei combine the continuous part of the diffusion and the accident jump process additively and they are able to allow prevention through intensity control of the jump process. This makes sense as many accidents are preventable through known measures. Our framework uses a different form of jump diffusion to enable the shutdown event to completely stop revenue generation in a continuous-time setting. This is done by building on a standard continuous Brownian-motion based output process. The main novelty is the multiplicative effect of the jump risk : upon the arrival of the risk, the whole of the output process comes to a halt. As an extension, we allow the halt to no longer be a complete fatality : production may continue at a degraded level through an investment by the principal.  From a methodological viewpoint, our reasoning uses a now standard method in dynamic contracting based on \cite{Sannikov08} and \cite{CPT:18} which consists in transforming both the first-best and second-best problems into classical Markovian control problems. The solution to these control problems can be characterized through a Hamilton-Jacobi-Bellman equation. Quite remarkably, this equation has, in our context, an explicit solution that is closely linked to a so-called Bernouilli ODE which facilitates many extensions. \\

\indent To the best of our knowledge, this paper is the first to explicitly introduce a default in a dynamic Principal-agent framework, in both a first-best (also called full Risk-Sharing) and second-best (also called Moral Hazard) setting. 
A key feature of our study is that the shape of the optimal contract is linear.
More precisely, the agent's compensation is the sum of two functions: the first is linear with respect to the output and serves to give the incentives, while the second is linear with respect to the effective duration of the contract and serves to share the default risk. While the linear incentive part of the contract is in line with the existing literature on continuous-time Principal-Agent problems without default under exponential utilities, the risk-sharing part deserves some clarification. 
The contract exposes both agents to a risk of exogenous interruption but it has two different regimes that are determined by an explicit relation between the risk-aversions and the agent's effort cost. Under the first regime, the agent is more sensitive to the risk of default than the principal. In this case, the principal deposits on the date $0$ a positive amount onto an escrow account whose balance will then decrease over time at a constant rate.  It is crucial to observe that the later the default arrives, the more the amount in the escrow account decreases to a point where it may even become negative. If the default occurs, the principal transfers the remaining balance to the agent. Under the second regime, the principal is more sensitive to the risk of default. In this case, the principal deposits a negative amount into the escrow account, which now grows at a constant rate and symmetrical reasoning applies. This linearity contrasts with the optimum obtained in \cite{CF} as the additive contribution of their jump process to revenue generation leads to a sub-linear wage.  This result is coherent with the paper by Hoffman and Pfeil \cite{HoffmanPfeil} which proves that, in line with the empirical studies by Bertrand and Mullainathan \cite{Bertrand}, the agent must be rewarded or punished  for a risk that is beyond his control.\\ 
Finally, this paper also explicitly characterizes the optimal contract when a possibility for shutdown risk mitigation exists at a cost. Such a possibility is coherent with agency-free external risk: prevention is not possible, at least on a short-term or medium-term time scale. At best the principal may be able to invest to mitigate its effects. Crucially we find that in many circumstances, investing is not optimal for the principal. When it is, it is only optimal up until some cutoff time related to a balance between the cost of investment, the agent’s rents  and possible remaining gain. \\

The rest of the document is structured as follows. In Section \ref{sec:mod}, we present the model and the Principal-Agent problems that we consider.  In Section \ref{FB}, we  analyse the first-best case where the principal observes the agent’s effort. Then in Section \ref{sec:OC}, we give our main results and analysis. In Section \ref{sec:M}, we extend our model to include a possibility for mitigation upon a halt. 

\section{The Model}\label{sec:mod}
The model is inherited from the classical work of Holmström and Milgrom \cite{HM}. A principal contracts with an agent to manage a project she owns. The agent influences the project’s profitability by exerting an unobservable effort. For a fixed effort policy, the output process is still random and the idiosyncratic uncertainty is modeled by a Brownian motion.\\
We assume that the contract matures at time $T>0$ and both principal and agent are risk-averse with CARA preferences. The departure from the classical model is as follows: we assume the project is facing some external risk that could partially or totally interrupt the production at some random time $\tau$. The probability distribution of $\tau$ is assumed to be independent of the Brownian motion that drives the uncertainty of the output process and also independent of the agent’s actions. 
Finally, we assume that the contract offers a transfer $W$ at time $T$ from the principal to the agent that is a functional of the output process.

\subsection{Probability setup}

Let $T > 0$ be some fixed time horizon. The key to modeling our Principal-Agent problems under an agency-free external risk of default is the simultaneous presence over the interval $[0,T]$ of a continuous random process and a jump process as well as the ability to extend the standard mathematical techniques used for dynamic contracting to this mixed setting. Thus, we shall deal with two kinds of information : the information from the output process, denoted as $\mathbb{F}=(\Fc_t)_{t \ge 0}$ and the information from the default time, i.e. the knowledge of the time where the default occurred in the past, if the default has appeared. This construction is not new and occurs frequently in mathematical finance\footnote{We refer the curious reader to the two important references \cite{AksamitJeanblanc} and \cite{JB}.}. \\

The complete probability space that we consider will be denoted as $(\Omega, \mathcal{G}, \mathbb{P}^0)$, with two independent stochastic processes : 
\begin{itemize}
\item $B$ a standard one-dimensional $\mathbb{F}$-Brownian motion,
\item $N$ the right-continuous single-jump process defined as $N_t = \textbf{1}_{\tau \leq t}$, $t$ in $[0,T]$ where $\tau$ is some positive random variable independent of $B$ that models the default time. 
\end{itemize}
$N$ will also be referred to as the default indicator process. We therefore use the standard approach of progressive enlargement of filtration by considering $\mathbb{G} = \left\{ \mathcal{G}_t, t \geq 0\right\}$  the smallest complete right-continuous extension of  $\mathbb{F}$ that makes $\tau$ a $\mathbb G$-stopping time. Because $\tau$ is independent of $B$, $B$ is a $\mathbb{G}$-Brownian motion under $\mathbb{P}^0$ according to Proposition 1.21 p 11 in \cite{AksamitJeanblanc}. We also suppose that there exists a bounded deterministic compensator of $N$, $\Lambda_t = \int_0^t \lambda(s)\,ds $ for some bounded function $\lambda(.)$ called the intensity implying that:
$$ M_t = N_t - \int_0^t \lambda(s) (1-N_s) ds, \quad t \in [0,T]$$
is a $\mathbb{G}$-compensated martingale. Note that through knowledge of the function $\lambda,$ the principal and agent can compute at time 0 the probability of default happening over the contracting period $[0,T]$. Indeed : 
$$ \mathbb{P}(\tau \leq T) = 1 - \exp(-\Lambda_T).$$

We first suppose for computational ease that the intensity $\lambda$ is a constant. We will see in Section \ref{deterministicintensity} that our results may easily be lifted to more general deterministic compensators. 

\begin{Remark}
Here we will suppose that the compensator of $N$ is common knowledge to both the Principal and the Agent. We could imagine settings where the Principal and Agent's beliefs regarding the risk of default may differ : this natural extension of our work would call for analysis of the dynamic contracting problem under hidden information which is left for future research.
\end{Remark}

\subsection{Principal-Agent Problem}

We suppose that the agent agrees to work for the principal over a time period $[0,T]$ and provide up to the default time a costly action  $(a_t)_{t \in [0,T]}$ belonging to $\mathcal A$, where $\mathcal A$ denotes the set of admissible $\mathbb{F}$-predictable strategies that will be specified later on. The Principal-Agent problem models the realistic setting where the principal cannot observe the agent's effort. As such the agent chooses his action in order to maximize his own utility. The principal must offer a wage based on the information driven by the output process up to the default time that incentivizes the agent to work efficiently and contribute positively to the output process. Mathematically, the unobservability of the agent's behaviour is modeled through a change of measure. Under $\mathbb{P}^0$ , we assume that the project’s profitability evolves as
$$ X_t := x_0 + \int_0^t (1-N_s) dB_s.$$
Thus, $\P^0$ corresponds to the probability distribution of the profitability when the agent  makes no effort over $[0,T]$. When the agent makes an effort $a=(a_t)_t$, we shall assume that the project’s profitability evolves as
$$ X_t := x_0 + \int_0^t a_s(1-N_s) ds + \int_0^t (1-N_s) dB^a_s,$$
where $B^a$ is a $\mathbb{F}$-Brownian motion under a measure $\mathbb{P}^a$.
The agent fully observes the decomposition of the production process under a measure $\mathbb{P}^a$ whilst the principal only observes the realization of $X_t$. 
 In order for the model to be consistent, the probabilities $\P^0$ and $\P^a$ must be equivalent for all $(a_t)_{t \in [0,T\wedge \tau]}$ belonging to $\mathcal A$. Therefore, we introduce the following set of actions
 $$
 \Bc=\Big\{ a=(a_t)_t : \quad\mathbb{F}\hbox{-predictable and taking values in } [-A,A] \text{ for some }A>0 \Big\}.
 $$
 The action process in $\Bc$ are uniformly bounded by some fixed constant $A>0$ that will be assumed as large as necessary.
 For $a \in \Bc$, we define $\P^a$ as
  $$ \frac{d\mathbb{P}^a}{d\mathbb{P}^0}|\mathcal{G}_T = \exp\left( \int_0^T a_s(1-N_s) dB_s - \frac{1}{2} \int_0^T |{a_s}|^2 (1-N_s)ds \right):=L_T.$$
 Because $\E^{0}(L_T)=1$ , $(B^a_t)_{t \in [0,T]}$ with $B^a_t = B_t - \int_0^t a_s (1-N_s)ds, t \in [0,T]$ is a $\mathbb{G}$-Brownian motion under $\mathbb{P}^a$ according to Proposition 3.6 c) p 55  in \cite{AksamitJeanblanc}. It is key to note that if halt occurs, i.e. if $\tau \leq T$, then the production process is halted before $T$ meaning that :
$ X^{a}_{t \wedge \tau} = X_t^a, \quad t \in [0,T]$. Let us then observe that an action $a=(a_t)_t$ of $\Bc$ can be extended to a $\mathbb{G}$-predictable process $(\tilde a_t)_{t \in [0,T]}$ by setting $\tilde a_t= a_t\ind_{t \le \tau}$.

  The cost of effort for the agent is modeled through a quadratic cost function : 
$ \kappa(a) :=  \kappa \frac{a^2}{2},$
for $\kappa>0$ some fixed parameter. As a reward for the agent's effort, the principal pays him a wage $W$ at time $T$.  $W$ is assumed to be a $\mathcal{G}_{T\wedge \tau}$ random variable which means that the payment at time $T$ in case of an early default is known at time $\tau$.
The principal and the agent are considered to be risk averse and risk aversion is modeled through two  CARA utility functions  : 
$$ U_P(x) := -\exp(-\gamma_P x) \; \text{and} \; U_A(x) := -\exp(-\gamma_A x) ,$$
where $\gamma_P > 0$ and $\gamma_A > 0$ are two fixed constants modeling the principal's and the agent's risk aversion. \\

In this setting and for any given wage $W $, the agent maximizes his own utility and solves : 
\begin{equation}
\label{eq:V0A}
V_0^A(W) = \sup_{a \in \mathcal{B}} \E^{a}\left[ U_A\left(W - \int_0^T \kappa(a_s(1-N_s)) ds \right) \right].
\end{equation}
A wage $W$ is said to be incentive compatible if there exists an action policy $a^*(W) \in \Bc$ that maximises (\ref{eq:V0A}) and thus satisfies
$$ V_0^A(W) =  \E^{a^*(W)}\left[ U_A\left(W - \int_0^T \kappa(a_s^*(W)(1-N_s)) ds \right) \right].$$

When the principal is able to offer an incentive compatible wage $W$, she knows what the agent's best reply will be. As such the principal establishes a set $\mathcal{A}^*(W) \subset \mathcal{B}$ of best replies for the agent for any incentive compatible $W$. Therefore, the first task is to characterize the set of incentive-compatible wages $\Wc_{IC}$. 
Only then may the principal consider maximizing his own utility by solving :
\begin{equation}
\label{eq:pbMHa}
\sup_{W \in \mathcal{W}_{IC}} \; \sup_{a^* \in \mathcal{A}^*(W)} \; \E^{{a^*}(W)}\left[U_P\left(X^{a^*{(W)}}_T-W\right)\right]
\end{equation}
under the participation constraint 
\begin{equation}
\label{eq:PC}
\E^{{a^*}(W)}\left[U_A\left(W-\int_0^T\kappa(a_s^*(W)(1-N_s))ds\right)\right] \geq U_A(y_{PC}),
\end{equation}
where $y_{PC}$ is a monetary reservation utility for the agent.

\begin{Remark}
\label{rem:stand}
Problem (\ref{eq:pbMHa}) has been thoroughly analyzed in a setting where the output process may not default (see the pioneer papers \cite{HM}, \cite{SS} ). Setting $\kappa=1$ for simplicity,  the optimal action is  constant and given by :
$$a^* =  \frac{\gamma_P+1}{\gamma_P+\gamma_A+1},$$
and the optimal wage is linear in the output:
$$W = y_{PC} + a^* X_T+  \left( \frac{\gamma_A-1}{2} ({a^*})^2  \right)T.$$
 We may naturally expect to encounter an extension of these results in this setting.
 \end{Remark}

\section{Optimal First-best contracting}
\label{FB}

We begin with analysis of the first-best benchmark (the  full Risk-Sharing problem) which leads to a simple optimal sharing rule. Of course this problem is not the most realistic when it comes to modeling  dynamic contracting situations. However it provides a benchmark to which we can compare the more realistic Moral Hazard situation. Indeed, the principal's utility in the full Risk-Sharing problem is the best that the principal will ever be able to obtain in a contracting situation as he may observe (and it is thus assumed that he may dictate) the agent's action.\\
 
To write the first-best problem, we assume that both the principal and the agent observe the variations of the same production process $(X_t^{a})_{t \in [0,T]}$ under $\P^0$:
 \begin{equation}
 \label{eq:prodprocess1}
 X^{a}_t := x_0 + \int_0^t a_s(1-N_s) ds + \int_0^t (1-N_s) dB_s. \quad t \in [0,T]
 \end{equation}
 The agent is guaranteed a minimum value of expected utility through the participation constraint : 
\begin{equation}
\label{eq:PCFB}
\E\left[U_A\left(W-\int_0^T\kappa(a_s(1-N_s))ds\right)\right] \geq U_A(y_{PC}), \\
\end{equation}
but has no further say on the wage or action. Consider the admissible set : 
$$\mathcal{A}_{PC} := \left\{  (W,a) \hbox{ such that } W \text{ is }\mathcal{G}_{T\wedge\tau} \text{ measurable with }\E\left[\exp(-2\gamma_A W)\right] < + \infty,  (a_t)_t \in \Bc, \, \text{ and  }(\ref{eq:PCFB}) \; \text{is satisfied}\right\}. $$

The full Risk-Sharing problem involves maximizing the principal's utility across $\mathcal{A}_{PC}$~:

\begin{equation}
\label{eq:prob}
\sup_{(W,a) \in \mathcal{A}_{PC}} \quad \E\left[U_P\left(X^{a}_T-W\right)\right].
\end{equation}

\subsection{Tackling the Participation Constraint}

A first step to optimal contracting in this first-best setting involves answering the following question:  can we characterize the set $\Ac_{PC}$? 
Following the standard route, we will first establish a necessary condition. For a given pair $(W,a) \in \Ac_{PC}$, let us introduce the agent’s continuation utility $(U^{(W,a)}_t)_t$ as follows:
$$ U^{(W,a)}_t := \E_t\left[ U_A\left( W - \int_t^T \kappa(a_s(1-N_s)) ds\right) \right], $$
where we use the shorthand notation : $\E_t[.] := \E[. | \mathcal{G}_t].$ 
We may write the Agent's continuation value process as the product  : 
 $$  U^{(W,a)}_t = \mathcal{M}_t^{(W,a)} \mathcal{D}_t^{(W,a)},$$
 where :
 $$  \mathcal{M}_t^{(W,a)} :=  \E_t\left[ U_A\left( W - \int_0^T \kappa(a_s(1-N_s)) ds\right) \right] \quad \text{and} \quad \mathcal{D}_t^{(W,a)} := \exp\left( -\gamma_A \int_0^t \kappa(a_s(1-N_s)) ds\right).$$
 Observe that for any admissible pair $(W,a)  \in \Ac_{PC}$, the process $\mathcal{M}=(\mathcal{M}_t^{(W,a)})_t$ is a $\mathbb{G}$-square integrable martingale.
 According to the Martingale Representation Theorem for $\mathbb{G}$-martingales (see \cite{AksamitJeanblanc}, Theorem 3.12 p. 60), there exists some predictable pair $(z_s, l_s)$ in $ \mathbb{H}^2 \times \mathbb{H}^2$,
 where $\mathbb{H}^2$ is the set of $\mathbb{F}$-predictable processes $Z$ with $\E\left[\int_0^T |Z_t|^2 dt\right] < + \infty$, such that : 
 $$  \mathcal{M}_t^{(W,a)}  :=   \mathcal{M}_0^{(W,a)} + \int_0^t z_s(1-N_s) dB_s + \int_0^t l_s(1-N_s) dM_s. $$
 Integration by parts yields the dynamic of $U$, noting that $\mathcal{D}$ has finite variation : 
 \begin{align*}
 d U_t^{(W,a)}  
 &= -\gamma_A \kappa(a_t(1-N_s) )  U_t^{(W,a)} dt + \mathcal{D}_t^{(W,a)} z_t(1-N_s)  dB_t +  \mathcal{D}_t^{(W,a)} l_t(1-N_s)  dM_t.
 \end{align*}
Setting $ Z_t^{(W,a)}  := \mathcal{D}_t^{(W,a)} z_t \in \mathbb{H}^2$ and $ K_t^{(W,a)} := \mathcal{D}_t^{(W,a)} l_t  \in \mathbb{H}^2$, we obtain: 
 \begin{align*}
 d U_t^{(W,a)}   
 &= -\gamma_A \kappa(a_t(1-N_s) )  U_t^{(W,a)} dt +Z_t^{(W,a)}(1-N_s)  dB_t +  K_t^{(W,a)}(1-N_s)   dM_t.
 \end{align*}
By construction, we have that $U_T^{(W,a)} = U_A(W)$. It follows that
$\left(U_t^{(W,a)}, Z_t^{(W,a)} , K_t^{(W,a)} \right)$ is a solution to the BSDE: 
 \begin{align}\label{BSDEFB}
  - d U_t^{(W,a)} = - Z_t^{(W,a)} (1-N_s) dB_t - K_t^{(W,a)} (1-N_s) dM_t  + \gamma_A \kappa(a_t(1-N_s) )  U_t^{(W,a)} dt,
 \end{align}
 with $U_T^{(W,a)} = U_A(W).$ Therefore,  (\ref{eq:PCFB}) is satisfied if and only if   $U_0^{(W,a)} \geq U_A(y_{PC})$.

\begin{Remark}\label{uniquenessBSDEFB}
Let $\mathbb{S}^2$ be the set of $\mathbb{G}-$adapted RCLL processes $U$ such that $$\E[\sup_{0 \leq t \leq T} | U_t|^2] < + \infty.$$
Through Proposition 2.6 of  \cite{Dumi}, the solution to \eqref{BSDEFB} is unique in $(\mathbb{S}^2\times\mathbb{H}^2\times \mathbb{H}^2)$. Indeed, the driver $g(\omega, U)=\gamma_A \kappa(a_t(1-N_t))  U$ is uniformly Lipschitz in $U$ because $(a_t)_t$ is bounded and the terminal condition is in $L^2$.
\end{Remark}

To sum up, we have the following necessary condition for admissibility.
\begin{Lemma}
If $(W,a) \in \Ac_{PC}$ then there exists a unique solution $\left(U_t^{(W,a)}, Z_t^{(W,a)} , K_t^{(W,a)} \right)$  in $(\mathbb{S}^2\times\mathbb{H}^2\times \mathbb{H}^2)$ to the BSDE \eqref{BSDEFB} such that $U_0^{(W,a)} \geq U_A(y_{PC})$.
\end{Lemma}

To obtain a sufficient condition, we introduce, for $\pi=(y_0, a, \beta, H ) \in \R\times \Bc \times \mathbb{H}^2\times \mathbb{H}^2 $, the wage process $( W_t^{\pi})_t$ defined as
\begin{align}\label{forwardFB}
W_t^{\pi} &:= y_0 + \int_0^t \beta_s (1-N_s)dB_s + \int_0^t H_s (1-N_s)dM_s + \int_0^t \left\{ \frac{\gamma_A}{2} \beta_s^2(1-N_s) + \kappa(a_s(1-N_s)) \right. \nonumber\\
&\left. \frac{\lambda}{\gamma_A} [\exp(-\gamma_A H_s) - 1 + \gamma_A H_s ](1-N_s) \right\} ds,
\end{align}
and consider the set 
$$\Gamma := \left\{ (y_0, a, \beta, H ) \in  \R\times \Bc \times \mathbb{H}^2\times \mathbb{H}^2 \text{ such that } y_0 \ge y_{PC} \text{ and } \E\left[ \exp(-2\gamma_AW_T^\pi)\right]<+\infty.   \right\}.$$

We have the following result.
\begin{Lemma}
\label{lemma:main1PB}
For any $\pi \in \Gamma$, the pair $(W_T^\pi,a)$ belongs to $\Ac_{PC}$.
\end{Lemma}
\begin{proof}
We apply It\^o’s formula to the process $Y_t^\pi=U_A(W_t^\pi)$ to obtain
$$
dY_t^\pi=-\gamma_A Y_t^\pi \beta_t(1-N_t) \,dB_t +Y_t^\pi\left( e^{-\gamma_AH_t}-1\right)(1-N_t)\,dM_t-\gamma_A\kappa(a_t(1-N_t))Y_t^\pi\,dt.
$$
Moreover, because $\pi \in \Gamma$, $Y_T^\pi=U_A(W_T^\pi)$ is square-integrable. Remark \ref{uniquenessBSDEFB} yields the triplet $\left(Y_t^\pi,-\gamma_AY_t^\pi\beta_t, Y_t^\pi ( e^{-\gamma_AH_t}-1)\right)$ is the unique solution 
in $(\mathbb{S}^2\times\mathbb{H}^2\times \mathbb{H}^2)$ to BSDE \eqref{BSDEFB} with terminal condition $U_A(W_T^\pi)$ when $\pi \in \Gamma$. Therefore,
$$
Y_0^\pi=U_A(y_0)=\E\left[ U_A\left(W_T^\pi-\int_0^T \kappa(a_s(1-N_s))\,ds\right)\right]\ge U_A(y_{PC}),
$$
and thus \eqref{eq:PCFB} is satisfied.
\end{proof}

\begin{Remark} \label{boundedcontrol}
The admissible contracts are essentially the terminal values of the controlled processes \eqref{forwardFB} for $\pi \in \Gamma$. The difficulty is that we do not know how to characterize the $\beta$ and $H$ processes that guarantee that $\pi$ belongs to  $\Gamma$. Nevertheless, it is easy to check by a standard application of the Gronwall lemma that if $\beta$ and $H$ are bounded then $\pi \in \Gamma$. This last observation will prove to be crucial in the explicit resolution of our problem.
\end{Remark}

\subsection{First-best Dynamic contracting}

Using Lemma \ref{lemma:main1PB}, the full Risk-Sharing problem under default writes as the Markovian control problem : 
 \begin{equation}
 \label{eq:pbparam}
V_P^{FB}:= \underset{\pi=(y_0, a,Z,K) \in \Gamma}{\text{sup}}\E\left[ U_P\left(X_T^{(x_0,a)} - W_T^{\pi}\right)\right],
 \end{equation}
 where  $X_t^{(x_0,a)}$ is given by :
 $$ dX^{(x_0,a)}_s =  a_s(1-N_s) ds +  (1-N_s) dB_s,$$ 
 with $X_0^{(x_0,a)}=x_0$
and the wage process is given by :
 \begin{align*}
 \label{eq:wageparam}
 dW_s^{\pi} =  Z_s (1-N_s)dB_s &+ K_s (1-N_s)dM_s \\
 &+  \left\{ \frac{\gamma_A}{2} Z_s^2 (1-N_s)+ \kappa(a_s(1-N_s)) + \frac{\lambda}{\gamma_A} [\exp(-\gamma_A K_s) - 1 + \gamma_A K_s ](1-N_s)\right\} ds,
 \end{align*}
 with $W_0^{\pi}=y_0$.\\
 We have the following key theorem for the first-best problem. 
 
 \begin{Theorem}
\label{theo:mainFB}
Let  $a^*_t = \dfrac{1}{\kappa}, Z^*_t = \dfrac{\gamma_P}{\gamma_P + \gamma_A},$ and let : 
 $$ K^*_t = \frac{1}{\gamma_P + \gamma_A} \log(\Phi_0(t)) ,$$
where :
\begin{equation}
\label{eq:phi}
\Phi_0(t) := \left(  \frac{c_1 + c_2}{c_1} \exp\left(c_1 \frac{\gamma_A}{\gamma_P + \gamma_A}(T-t)\right) - \frac{c_2}{c_1}\right)^{\frac{\gamma_P+\gamma_A}{\gamma_A}} ,
\end{equation}
with :
$$ c_1 := \frac{\gamma_P^2 \gamma_A}{2(\gamma_P+\gamma_A) } - \frac{\gamma_P}{2\kappa} - \lambda \frac{\gamma_P + \gamma_A}{\gamma_A} \quad \text{and} \quad c_2 := \lambda \frac{\gamma_P + \gamma_A}{\gamma_A}.\\$$

Then $\pi^*=(y_{PC},a^*,Z^*,K^*) \in \Gamma$ parameterizes the optimal contract $(W^{\pi^*}_{T},a^*)$ for the first-best problem. 
\end{Theorem}

The rest of this subsection is dedicated to the proof of this Theorem.  We first make the following observation. As $X$ remains constant after $\tau$, the principal has no further decision to make after the default time. Thus, its value function is constant and equal to $ U_P(x-y)$ on the interval $[\tau,T]$.\\ 
We now focus on the control part of the problem (i.e. computation of the optimal control triplet $\tilde \pi=(a,Z,K)$ for a given pair $(x_0,y_0)$). To do so, we follow the dynamic programming approach developed in \cite{Pham:10}, Section 4
to define the value function  
\begin{equation}\label{StochasticControlFB}
V(0,x_0,y_0)=\sup_{\tilde \pi \in \tilde \Gamma} \E\left[ U_P(X_{T}^{a} - W_{T}^{\tilde \pi} )(1-N_T)+\int_0^T U_P(X_t^a-W_t^{\tilde \pi})\lambda e^{-\lambda t}\,dt \right],
\end{equation}
where
$$
\tilde \Gamma = \left\{ \tilde \pi \in \mathcal{B} \times   \mathbb{H}^2 \times \mathbb{H}^2\right\},$$
Because $\Gamma \subset \R\times\tilde\Gamma$, we have 
$$ V_P^{FB} \le \sup_{y_0 \geq y_{PC}} V(0,x_0,y_0).$$

According to stochastic control theory, the Hamilton-Jacobi-Bellman equation associated to the stochastic control problem \eqref{StochasticControlFB} is the following (see \cite{Oksendal}): 
\begin{align}
\label{eq:HJBFB}
\partial_t v(t,x,y) + \sup_{a,Z,K} \left\{ \partial_xv(t,x,y)a + \partial_yv(t,x,y) \left[\frac{\gamma_A}{2} Z^2 + \kappa(a) +  \frac{\lambda}{\gamma_A} [\exp(-\gamma_A K) - 1]\right] \right. \nonumber\\
\left.+ \lambda \left[ U_P(x-y-K) - v_0(t,x,y)\right]+ \partial_{yy} v(t,x,y) \frac{Z^2}{2} + \frac{1}{2} \partial_{xx} v(t,x,y)+ \partial_{xy}v(t,x,y) Z 
 \right\} = 0,
\end{align}
with the boundary condition :
$$v(T,x,y) =  U_P(x-y).$$

It happens that the HJB equation \eqref{eq:HJBFB} is explicitly solvable by exploiting the separability property of the exponential utility function. 
\begin{Lemma}
\label{lem:V0RSFB}
The function $v(t,x,y)=U_P(x-y) \Phi_0(t)$
with~:
$$\Phi_0(t) = \left(  \frac{c_1 + c_2}{c_1} \exp\left(c_1 \frac{\gamma_A}{\gamma_P + \gamma_A}(T-t)\right) - \frac{c_2}{c_1}\right)^{\frac{\gamma_P+\gamma_A}{\gamma_A}} ,$$
where~:  
$$ c_1 = \frac{\gamma_P^2 \gamma_A}{2(\gamma_P+\gamma_A) } - \frac{\gamma_P}{2\kappa} - \lambda \frac{\gamma_P + \gamma_A}{\gamma_A} \quad \text{and} \quad c_2 = \lambda \frac{\gamma_P + \gamma_A}{\gamma_A},$$
solves (in the classical sense) the HJB partial differential equation (\ref{eq:HJBFB}). \\
Furthermore $a^*_t = \dfrac{1}{\kappa}, Z^*_t = \dfrac{\gamma_P}{\gamma_P + \gamma_A}$ and $K^*_t = \dfrac{1}{\gamma_P + \gamma_A} \log(\Phi_0(t))$ are the optimal controls.
\end{Lemma}
\begin{proof}
We search for a solution to Equation (\ref{eq:HJBFB})  for a $v$ of the form : $$v(t,x,y)=U_P(x-y) \Phi_0(t),$$ with $\Phi_0$ a positive mapping. Such a $v$ satisfies  (\ref{eq:HJBFB})  if and only if  $\Phi_0(t)$ solves the PDE : 
\begin{align*}
 \Phi_0'(t) + \inf_{a,Z,K}\left\{ -\gamma_P \Phi_0(t)a + \gamma_P \Phi_0(t) \left( \frac{\gamma_A}{2}Z^2 + \kappa(a) + \frac{\lambda}{\gamma_A} \left\{\exp(-\gamma_A K) -1 \right\} \right) \right. \\
 \left. + \gamma_P^2 \Phi_0(t) \frac{Z^2}{2} + \frac{\gamma_P^2}{2} \Phi_0(t) - \gamma_P^2 \Phi_0(t) Z + \lambda \left( \exp(\gamma_P K) - \Phi_0(t) \right)\right\}=0,
 \end{align*}
with the boundary condition $\Phi_0(T)=1.$  As $\Phi_0$ is a positive mapping, the infimum is well defined. We derive the following first order conditions that must be satisfied by the optimal controls : 
$$ \begin{cases}
\gamma_P \Phi_0(t) = \gamma_P \kappa a \Phi_0(t)\\
\gamma_P \Phi_0(t) Z (\gamma_A  + \gamma_P) = \gamma_P^2 \Phi_0(t)\\
\gamma_P \Phi_0(t) \lambda \exp(-\gamma_A K) = \gamma_P \lambda \exp(\gamma_P K),
\end{cases}$$
equating to : 
$$ a^* = \frac{1}{\kappa}, \quad Z^* = \frac{\gamma_P}{\gamma_P + \gamma_A}, \quad K^* = \frac{\log(\Phi_0(t))}{\gamma_P + \gamma_A}.$$
It follows that : 
\begin{align*}
&\inf_{a,Z,K}\left\{ -\gamma_P \Phi_0(t)a + \gamma_P \Phi_0(t) \left( \frac{\gamma_A}{2}Z^2 + \kappa(a) + \frac{\lambda}{\gamma_A} \left\{\exp(-\gamma_A K) -1 \right\} \right) \right. \\
 &\left. + \gamma_P^2 \Phi_0(t) \frac{Z^2}{2} + \frac{\gamma_P^2}{2} \Phi_0(t) - \gamma_P^2 \Phi_0(t) Z + \lambda \left( \exp(\gamma_P K) - \Phi_0(t) \right)\right\}\\
 &=   -\gamma_P \Phi_0(t)a^* + \gamma_P \Phi_0(t) \left( \frac{\gamma_A}{2}{Z^*}^2 + \kappa(a^*) + \frac{\lambda}{\gamma_A} \left\{\exp(-\gamma_A K^*) -1 \right\} \right)  \\
 &+ \gamma_P^2 \Phi_0(t) \frac{{Z^*}^2}{2} + \frac{\gamma_P^2}{2} \Phi_0(t) - \gamma_P^2 \Phi_0(t) Z^* + \lambda \left( \exp(\gamma_P K^*) - \Phi_0(t) \right) \\
 &=  \underbrace{\Phi_0(t) \frac{\gamma_P^2\gamma_A}{2(\gamma_P+\gamma_A)}}_{\text{terms with } Z^*}  \underbrace{- \Phi_0(t) \frac{\gamma_P}{2\kappa}}_{ \text{terms with }a^* }  \underbrace{- \lambda  \frac{\gamma_P + \gamma_A}{\gamma_A} \Phi_0(t) + \lambda \frac{\gamma_P + \gamma_A}{\gamma_A} \Phi_0(t)^{\frac{\gamma_P}{\gamma_P+\gamma_A}}.}_{ \text{terms with }K^*}
 \end{align*}

We may inject this expression back into the PDE on $\Phi_0$. Doing so yields the following Bernoulli equation : 
$$ \Phi'_0(t) + c_1\Phi_0(t) + c_2 \Phi_0(t)^{\frac{\gamma_P}{\gamma_P+\gamma_A}}=0, \quad \Phi_0(T) = 1,$$
where 
$$ c_1 = \frac{\gamma_P^2 \gamma_A}{2(\gamma_P+\gamma_A) } - \frac{\gamma_P}{2\kappa} - \lambda \frac{\gamma_P + \gamma_A}{\gamma_A} \quad \text{and} \quad c_2 = \lambda \frac{\gamma_P + \gamma_A}{\gamma_A}.$$

The unique solution to this equation is (see for instance \cite{Zwillinger:97}) : 
$$\Phi_0(t) = \left(  \frac{c_1 + c_2}{c_1} \exp\left(c_1 \frac{\gamma_A}{\gamma_P + \gamma_A}(T-t)\right) - \frac{c_2}{c_1}\right)^{\frac{\gamma_P+\gamma_A}{\gamma_A}},$$
and the result follows. 
\end{proof}

\begin{proof}[\textbf{\scshape{Proof of Theorem \ref{theo:mainFB}}}]
The value function $v(t,x,y)=U_P(x-y) \Phi_0(t)$ is a classical solution to the HJB equation (\ref{eq:HJBFB}). A standard verification theorem yields  that $v=V$. 
Through Lemma \ref{lem:V0RSFB}, the optimal controls for the full Risk-Sharing problem are : 
$$a^*_t = \dfrac{1}{\kappa}, Z^*_t = \dfrac{\gamma_P}{\gamma_P + \gamma_A}\; \text{and} \; K^*_t = \dfrac{1}{\gamma_P + \gamma_A} \log(\Phi_0(t)),$$
with $\Phi_0$ as defined in Lemma \ref{lem:V0RSFB}. These controls are free of $y$ and it follows that~: 
$$V(0,x_0,y_0) = E\left[ U_P\left(X_T^{(x_0,a^*)} - W_T^{(y_0,a^*,Z^*,K^*)}\right)\right],$$
is a decreasing function of $y_0$. Thus we obtain 
$$\sup_{y_0\ge y_{PC}}V_0(0,x_0,y_0) = E\left[ U_P\left(X_T^{(x_0,a^*)} - W_T^{(y_{PC},a^*,Z^*,K^*)}\right)\right].$$
Finally, we observe that the optimal controls are bounded and thus Remark \eqref{boundedcontrol} yields $\pi^*=(y_{PC},a^*,Z^*,K^*) \in \Gamma$. As a consequence,
$$
\sup_{y_0\ge y_{PC}} V(0,x_0,y_0) =  \E\left[ U_P\left(X_T^{(x_0,a)} - W_T^{(y_{PC},a^*,Z^*,K^*)}\right)\right] \le V_P^{FB}.
$$
Because the reverse inequality holds, the final result follows. 
\end{proof}

\section{Optimal contracting under shutdown risk}
\label{sec:OC}

\subsection{Main results}

The following is dedicated to our main result for the Moral Hazard problem. We shall state our main theorem with the explicit optimal contract before turning to some analysis of the effect of the shutdown on dynamic contracting.
In the case of moral hazard, one is forced to make a stronger assumption about the nature of a contract. This stronger hypothesis will naturally appear to justify the martingale optimality principle. In our setting, a contract is a $\mathbb{G}_{T\wedge \tau}$ measurable random variable $W$  such that for every $\beta \in \R$,  we have
$$
\E\left[ \exp(\beta W)\right]<+\infty.
$$   

A first step to optimal contracting involves answering the preliminary question:  can we characterize incentive compatible wages and if so what is the related optimal action for the agent?
The characterization of incentive compatible contracts relies on the martingale optimality principle (see \cite{HIM:2005} and \cite{Rouge:2000}) that we recall below.
\begin{Lemma}[Martingale Optimality Principle]
\label{lem:MOP}
Given a contract $W$, consider a family of stochastic processes $R^{a}(W) := (R_t^a)_{t \in [0,T]}$ indexed by $a$ in $\mathcal{B}$ that satisfies : 
\begin{enumerate}
\item $R_T^a = U_A(W - \int_0^T \kappa(a_s(1-N_s)) ds)$ for any $a$ in $\mathcal{B}$
\item $R^a_.$ is a $\mathbb{P}^{a}$-supermartingale  for any $a$ in $\mathcal{B}$
\item $R^a_0$ is independent of $a$.
\item There exists $a^*$ in $\mathcal{B}$ such that $R^{a^*}$ is a  $\mathbb{P}^{a^*}$-martingale.
\end{enumerate}
Then,
$$
R_0^{a^*}=\E^{a^*}\left[ U_A(W - \int_0^T \kappa(a^*_s) ds)\right] \ge \E^{a}\left[ U_A(W - \int_0^T \kappa(a_s(1-N_s)) ds)\right],
$$
meaning that $a^*$ is the optimal agent’s action in response to the contract $W$. 
\end{Lemma}

We  will construct such a family following the standard route. Consider a given contract $W$, we define the family $R^{a}(W) := (R_t^a)_{t \in [0,T]}$ by
$$
 R_t^a := -\exp\left(-\gamma_A \left( Y_t(W) - \int_0^t \kappa(a_s(1-N_s)) ds \right)\right),
$$
where $(Y(W),Z(W),K(W))$ in $(\mathbb{S}^2\times\mathbb{H}^2\times \mathbb{H}^2)$ is the unique solution of the following BSDE under $\P^0$
\begin{equation}
\label{BSDESB}
Y_t(W)=W-\int_{t}^{T} f(Z_s(W),K_s(W))(1-N_s)\,ds-\int_{t}^{T} Z_s(W)(1-N_s)\,dB_s-\int_{t}^{T} K_s(W)(1-N_s)\,dM_s,
\end{equation}
with
$$
f(z,k) :=  \frac{1}{2} \gamma_A  z^2 + \lambda k  +  \frac{\lambda}{\gamma_A}(e^{-\gamma_A k}-1) + \inf_{a \in \mathcal{B}} \left\{ \kappa(a) - a z \right\}.
$$
\begin{Remark} The theoretical justification of the well-posedness of the BSDE \eqref{BSDESB} deserves some comments. The first results were obtained in \cite{KLN:13} and \cite{JMPR:15} when the contract $W$ is assumed to be bounded.
The necessary extension in our model when $W$ admits an exponential moment has been treated recently in the paper \cite{Martin2020}.
\end{Remark}

By construction, $R_T^a = U_A(W - \int_0^T \kappa(a_s(1-N_s)) ds)$ for any $a$ in $\mathcal{B}$. Moreover, $R^a_0=Y_0(W)$ is independent of the agent’s action $a$.
We compute the variations of $R^a$ and obtain :  
\begin{align*}
&=- \gamma_A R_s^a Z_s (1-N_s)dB_s + R_s^a(e^{-\gamma_A K_s}-1)(1-N_s) dM_s\\
&+ R_s^a \gamma_A \left\{ \frac{1}{2} \gamma_A  Z_s^2 -  f(Z_s,K_s) + \kappa(a_s(1-N_s)) + \lambda K_s  +  \frac{\lambda}{\gamma_A}(e^{-\gamma_A K_s}-1) \right\}(1-N_s)ds.\\
&=- \gamma_A R_s^a Z_s(1-N_s) dB_s^a + R_s^a(e^{-\gamma_A K_s}-1)(1-N_s) dM_s\\
&+ R_s^a \gamma_A \left\{ \frac{1}{2} \gamma_A  Z_s^2 -  f(Z_s,K_s) + \kappa(a_s(1-N_s)) + \lambda K_s  +  \frac{\lambda}{\gamma_A}(e^{-\gamma_A K_s}-1) -a_s Z_s \right\}(1-N_s)ds.
\end{align*}
Thus $R^a$ is a $\mathbb{P}^{a}$-super-martingale for every $a$ in $\mathcal{B}$, the function
$$
a^*(z)=-A\ind_{z\le -\kappa A}+\frac{z}{\kappa}\ind_{-\kappa A \le z\le \kappa A}+A\ind_{z\ge \kappa A}
$$
is a unique minimizer for $f$ and $R^{a^*}$ is a $\mathbb{P}^{a^*}$-martingale. As a consequence, every contract $W$ is incentive compatible which a unique best reply $a^*(Z(W))$. Finally, a contract $W$ satisfies the participation constraint if
and only if $Y_0(W) \ge y_{PC}$.\\

Relying on the idea of Sannikov \cite{Sannikov08} and its recent theoretical justification by Cvitanic, Possamai and Touzi \cite{CPT:18}, we will consider the agent promised wage $Y(W)$ as a state variable to embed the principal’s problem into the class of Markovian problems, by considering the sensitivities of the agent’s promised wage $Z(W)$ and $K(W)$ as control variables. For $\pi=(y_0,Z,K) \in [y_{PC} ; + \infty) \times \mathbb{H}^2 \times  \mathbb{H}^2$, we define
under $\P^0$, the control process called the agent continuation value
\begin{equation}
 \label{eq:wageprocess2}
 W^{(y_0,Z,K)}_t =  y_0 +  \int_0^t Z_s (1-N_s)dB_s + \int_0^ t K_s (1-N_s)dM_s + \int_0^t f(Z_s,K_s)(1-N_s)\, ds.
 \end{equation}
Under $\P^*:=\P^{(a^*(Z))}$, we thus have
\begin{align}\label{eq:wageprocess2bis}
W^{(y_0,Z,K)}_t &=  y_0 +  \int_0^t Z_s(1-N_s) dB_s^* + \int_0^ t K_s(1-N_s) dM_s \\ \nonumber
&+ \int_0^t \left\{ \frac{\gamma_A}{2}   Z_s^2 + \kappa(a^*(Z_s))+ \frac{\lambda}{\gamma_A} [\exp(-\gamma_A K_s) - 1 + \gamma_A K_s ] \right\} (1-N_s)ds \\ \nonumber
&=  y_0 +  \int_0^t Z_s(1-N_s) dB_s^* + \int_0^ t K_s(1-N_s) dN_s \\ 
&+ \int_0^t \left\{ \frac{\gamma_A}{2}   Z_s^2 + \kappa(a^*(Z_s))+ \frac{\lambda}{\gamma_A} [\exp(-\gamma_A K_s) - 1 ] \right\} (1-N_s)ds  \nonumber
\end{align}
Now, we consider the set
$$
\zeta = \left\{ \pi=(Z,K) \in  \mathbb{H}^2 \times  \mathbb{H}^2 \text{ such that } \forall \beta \in \R,\,\E\left[  \exp(\beta W_T^{(y,Z,K)})  \right]<+\infty \text{ for }y \in \R \right\}.
$$
By construction, $W^{(y,\pi)}_{T}$ is a contract  that satisfies the participation constraint for every $\pi \in \zeta$ and $y \ge y_{PC}$. Moreover, by the well-posedness of the BSDE \eqref{BSDESB} , every contract $W$ that satisfies the participation constraint can be written $W^{(Y_0(W),Z(W),K(W))}_{T}$
with $\pi(W)=(Z(W),K(W)) \in \zeta$. Therefore, the problem of the principal can now be rewritten as the following optimisation problem
$$
V_P:=\sup_{y \ge y_{pc}} v(0,x,y),
$$
where

\begin{equation}
\label{MarkovSB}
v(0,x,y)=\sup_{\pi \in \zeta} \E^*\left[U_P(X_{T\wedge\tau}-W_{T\wedge\tau}^\pi) \right]
\end{equation}

To characterize the optimal contract, we will proceed analogously as in the full risk sharing case by constructing a smooth solution to the HJB equation associated to the Markov control problem \eqref{MarkovSB} given by

\begin{align}
\label{eq:HJB2}
0=\partial_t v(t,x,y) + \inf_{Z,K} \left\{ \partial_xv(t,x,y)\dfrac{Z}{\kappa} + \partial_yv(t,x,y) \left[\frac{\gamma_A}{2} Z^2 + \kappa(a^*(Z)) + \frac{\lambda}{\gamma_A} [\exp(-\gamma_A K) - 1]\right] \right. \nonumber\\
\left.+ \lambda \left[ U_P(x-y-K) - v(t,x,y)\right]+ \partial_{yy} v_0(t,x,y) \frac{Z^2}{2} + \frac{1}{2} \partial_{xx} v(t,x,y)+ \partial_{xy}v(t,x,y) Z 
 \right\},
\end{align}

\begin{Lemma}
\label{lem:temp2}
Assume the constant $A$ in the definition of the set of admissible efforts ${\cal B}$ satisfies 
$$
A > \dfrac{\gamma_P + \kappa^{-1}}{\kappa(\gamma_P + \gamma_A) + 1}\,.
$$
Then, the function $U_P(x-y) \Phi_0(t),$
with 
$$\Phi_0(t) = \left(  \frac{c_1 + c_2}{c_1} \exp\left(c_1 \frac{\gamma_A}{\gamma_P + \gamma_A}(T-t)\right) - \frac{c_2}{c_1}\right)^{\frac{\gamma_P+\gamma_A}{\gamma_A}},$$
where 
$$  c_1 = \frac{\gamma_P^2 \gamma_A}{2(\gamma_P+\gamma_A + \kappa^{-1}) } - \frac{\gamma_P \kappa^{-1}(\gamma_P+\kappa^{-1})}{2{(\gamma_P + \gamma_A + \kappa^{-1})}} - \lambda \frac{\gamma_P + \gamma_A}{\gamma_A} \quad \text{and} \quad c_2 = \lambda \frac{\gamma_P + \gamma_A}{\gamma_A}.$$
solves in the classical sense the HJB equation (\ref{eq:HJB2}). In particular $Z^*_t = \dfrac{\gamma_P + \kappa^{-1}}{\gamma_P + \gamma_A + \kappa^{-1}}$ and $K^*_t = \dfrac{1}{\gamma_P + \gamma_A} \log(\Phi_0(t)),$
\end{Lemma}
\begin{proof}
Because the assumption on $A$ implies $a^*(z)=z/\kappa$, the proof of this lemma is a direct adaptation of the proof of Lemma \ref{lem:V0RSFB} to which we refer the reader. 
\end{proof}

We are in a position to prove the main result of this section

\begin{Theorem}
\label{theo:main1}
We have the following explicit characterizations of the optimal contracts. Let $A$ as in the Lemma \ref{lem:temp2} and
let $Z^*_t = \dfrac{\gamma_P + \kappa^{-1}}{\gamma_P + \gamma_A + \kappa^{-1}}$ and 
$ K^*_t = \frac{1}{\gamma_P + \gamma_A} \log(\Phi_0(t)),$
where $\Phi_0$ is defined as in (\ref{eq:phi}) with the constants :  
$$  c_1 := \frac{\gamma_P^2 \gamma_A}{2(\gamma_P+\gamma_A + \kappa^{-1}) } - \frac{\gamma_P \kappa^{-1} (\gamma_P+\kappa^{-1})}{2{(\gamma_P + \gamma_A + \kappa^{-1})}} - \lambda \frac{\gamma_P + \gamma_A}{\gamma_A} \quad \text{and} \quad c_2 := \lambda \frac{\gamma_P + \gamma_A}{\gamma_A}.\\$$

Then $(y_{PC},Z^*,K^*)$ parametrizes the optimal wage for the Moral Hazard problem. The Agent performs the optimal action $\frac{Z^*}{\kappa}$. 

\end{Theorem}

\begin{proof}
Because the function $U_P(x-y) \Phi_0(t)$ is a classical solution to the HJB equation (\ref{eq:HJB2}) and the optimal controls are bounded and free of $y$, we proceed analogously as in the proof of Theorem \ref{theo:mainFB}.
Finally, we have to prove that the optimal wage $W^*=Y_T^{(y_{PC},Z^*,K^*)}$ admits exponential moments to close the loop. According to \eqref{eq:wageprocess2bis}, we have
$$
W^*=y_{PC}+Z^*B^*_{T\wedge \tau}+\frac{1}{2}\left(\gamma_A+\frac{1}{\kappa}\right)(Z^*)^2 (T\wedge \tau )+K^*_{\tau_{-}} \textbf{1}_{\tau \leq T}  + \int_0^T  \frac{\lambda}{\gamma_A} [\exp(-\gamma_A K^*_s) - 1 ](1-N_s) ds. 
$$
Because $(B^*_t)_t$ is a Brownian motion and $K^*_t$ is deterministic, it is straightforward to check that $W^*$ admits exponential moments.
\end{proof}

\subsection{Model analysis}

The optimal contract includes two components. One is  linear in the output with an incentivizing slope  that is similar to the classical optimal contract found in \cite{HM}. This is necessary to implement a desirable level of effort. The other  is unrelated to the incentives but linked to the shutdown risk sharing. It is key to observe that this second term is nonzero even if the shutdown risk does not materialize before the termination of the contract. \\
The characterization of the optimal contracts in Theorem \ref{theo:main1} sparks an immediate observation:  the two parties only need to be committed to the contracting agreement up until $T \wedge \tau$.  Therefore in this simple model, using an expected-utility related reasoning and without considering mechanisms such as employment law, the occurence of the agency-free external risk, halting production, leads to early contract terminations. This is in line with what actually happened during the Covid pandemic. Indeed in the USA and in eight weeks of the pandemic, 36.5 million people applied for unemployment insurance. In more protective economies, mass redundancies were only prevented through the instauration of furlough type schemes allowing private employees' wages to temporarily be paid by gouvernements. This phenomena makes fundamental sense : a principal whose output process is completely halted cannot enforce the agent to work hard because she has no revenue to provide the incentives. Let's focus on the second term:   

\begin{equation}
\label{extraK}
 K^*_{\tau_{-}} \textbf{1}_{\tau \leq T}  + \int_0^T  \frac{\lambda}{\gamma_A} [\exp(-\gamma_A K^*_s) - 1 ](1-N_s) ds, 
\end{equation}

Understanding the effect of these extra terms is crucial to fully understand the sharing of the agency-free shutdown risk.
First, we show that the sign of the control $K^*$ is constant.

\begin{Lemma}
\label{lemma:K}
Let $c_1$ and $c_2$ be the relevant constants given in Theorem \ref{theo:main1} then the optimal control $(K^*_t)_{t \in [0,T]}$ can be expressed as : 
\begin{equation}\label{Kversion2}
 K_t^* =  \frac{1}{\gamma_A} \log\left(\E\left[\exp\left(\frac{\gamma_A}{\gamma_P+\gamma_A}(c_1+c_2)((T-t) \wedge \tau)\right)\right]\right)\quad t \in [0,T]. 
 \end{equation}
\end{Lemma}
\begin{proof}
We have that : 
$$K^*_t = \dfrac{1}{\gamma_P + \gamma_A} \log(\Phi_0(t)) ,$$ 
with $$\Phi_0(t) = \left(  \frac{c_1 + c_2}{c_1} \exp\left(c_1 \frac{\gamma_A}{\gamma_P + \gamma_A}(T-t)\right) - \frac{c_2}{c_1}\right)^{\frac{\gamma_P+\gamma_A}{\gamma_A}}.$$
The aim here is to link this expression for $\Phi_0$ to that of an expected value. As such, we consider the following expected value that decomposes as shown  : 
\begin{align*}
&\E\left[\exp\left(\frac{\gamma_A}{\gamma_P+\gamma_A}(c_1+c_2)((T-t) \wedge \tau)\right)\right] \\
&= \E\left[ \exp\left(\frac{\gamma_A}{\gamma_P+\gamma_A}(c_1+c_2)(T-t) \right) \textbf{1}_{\tau > T-t}\right]+ \E\left[ \exp\left(\frac{\gamma_A}{\gamma_P+\gamma_A}(c_1+c_2)\tau \right) \textbf{1}_{\tau \leq T-t}\right].\\
\end{align*}
Using $c_2 = \dfrac{\gamma_P + \gamma_A}{\gamma_A} \lambda$, the first term of the expected value rewrites as follows : 
\begin{align*} \E\left[ \exp\left(\frac{\gamma_A}{\gamma_P+\gamma_A}(c_1+c_2)(T-t) \right) \textbf{1}_{\tau > T-t}\right] 
&=  \exp\left(\frac{\gamma_A}{\gamma_P+\gamma_A}(c_1+c_2)(T-t) \right) \exp\left(- \lambda(T-t) \right)  \\
&=   \exp\left(c_1 \frac{\gamma_A}{\gamma_P+\gamma_A} (T-t) \right).
\end{align*}
It remains to compute the second term. We obtain :  
\begin{align*}
 \E\left[ \exp\left(\frac{\gamma_A}{\gamma_P+\gamma_A}(c_1+c_2)\tau \right) \textbf{1}_{\tau \leq T-t}\right] &= \int_0^{T-t} \lambda \exp\left( \frac{\gamma_A}{\gamma_P+\gamma_A}(c_1+c_2)s            \right)  \exp\left(- \lambda s\right) ds\\
 &=\int_0^{T-t} \lambda \exp\left( \frac{\gamma_A}{\gamma_P+\gamma_A}(c_1+c_2)s            - \lambda s\right) ds\\
 &= \int_0^{T-t} \lambda \exp\left( \frac{\gamma_A}{\gamma_P+\gamma_A} c_1 s           \right) ds\\
 &= \left[ \frac{\lambda}{c_1}\frac{\gamma_P + \gamma_A}{\gamma_A}  \exp\left( \frac{\gamma_A}{\gamma_P+\gamma_A} c_1 s           \right)\right]_0^{T-t}\\
  &= \left[ \frac{c_2}{c_1}  \exp\left( \frac{\gamma_A}{\gamma_P+\gamma_A} c_1 s           \right)\right]_0^{T-t}\\
 &= \frac{c_2}{c_1}   \exp\left( c_1 \frac{\gamma_A}{\gamma_P+\gamma_A} (T-t)           \right) -  \frac{c_2}{c_1} . 
\end{align*}
Combining both terms we reach the final expression : 
\begin{align*}
\E\left[\exp\left(\frac{\gamma_A}{\gamma_P+\gamma_A}(c_1+c_2)((T-t) \wedge \tau)\right)\right]  
&=  \frac{c_1 + c_2}{c_1} \exp\left(c_1 \frac{\gamma_A}{\gamma_P + \gamma_A}(T-t)\right) - \frac{c_2}{c_1}. 
\end{align*}
Therefore we identify that : 
$$ \Phi_0(t) = \left(\E\left[\exp\left(\frac{\gamma_A}{\gamma_P+\gamma_A}(c_1+c_2)((T-t) \wedge \tau)\right)\right]\right)^{\frac{\gamma_P+\gamma_A}{\gamma_A}}.$$
As a consequence, we may also rewrite $K^*_t.$ Indeed : 
 $$ K^*_t = \frac{1}{\gamma_P + \gamma_A} \log(\Phi_0(t)) ,$$
and with the new expression for $\Phi_0$ we obtain the result : 
$$ K_t^* =  \frac{1}{\gamma_A} \log\left(\E\left[\exp\left(\frac{\gamma_A}{\gamma_P+\gamma_A}(c_1+c_2)((T-t) \wedge \tau)\right)\right]\right).$$
\end{proof}

\begin{Remark}
We have the same expression for the optimal control $K^*$ in the first-best case, using for $c_1$ and $c_2$ the relevant constants given in Theorem \ref{theo:mainFB}.
\end{Remark}
As a consequence, this alternative form for $K^*$ leads to easy analysis of the sign of the control, given in the following lemma. 

\begin{Lemma}
\label{lemma:signK}The sign of $K^*$ over the contracting period $[0,T]$ is constant and entirely determined by the model's risk aversions $\gamma_P$ and $\gamma_A$, and the Agent's effort cost $\kappa.$ Indeed, 
the sign of $K^*$ is equal to the sign of  $\gamma_P \gamma_A - \gamma_P \kappa^{-1} - (\kappa^{-1})^2.$  
Moreover, $K^*_t$ varies monotonously in time, with $K^*_T = 0.$
\end{Lemma}
\begin{proof}
From the expression \eqref{Kversion2}, we easily deduce that :
\begin{itemize}
\item If $c_1+c_2=0$ then $K^*_t=0$ for every $t \in [0,T]$,
\item if $c_1+c_2>0$ then $K^*_t >0$ for $t \le \tau$ and the function $t \to K^*_t$ decreases,
\item if $c_1+c_2<0$ then $K^*_t <0$ for $t \le \tau$ and the function $t \to K^*_t$ increases.
\end{itemize}
Replacing $c_1$ and $c_2$ by their relevant expressions in each case leads to the result. \\
\end{proof}
Finally, we will show that the risk-sharing component of the contract is in fact linear with respect to the default time. This is a strong result of our study for which we had no ex-ante intuition.
\begin{Corollary}
The shutdown risk-sharing component of the optimal wage is linear in the default time. More precisely, the optimal wage is
$$
W^*=y_{PC}+Z^*B^*_{T\wedge \tau}+\frac{1}{2}\left(\gamma_A+\frac{1}{\kappa}\right)(Z^*)^2 (T\wedge \tau ) +K^*_0  -\left( \frac{c_1}{\gamma_P + \gamma_A} + \frac{\lambda}{\gamma_A}\right)(T \wedge \tau).
$$
\end{Corollary}
\begin{proof}
Because $K^*_T=0$, the optimal wage can be written
$$
W^*=y_{PC}+Z^*B^*_{T\wedge \tau}+\frac{1}{2}\left(\gamma_A+\frac{1}{\kappa}\right)(Z^*)^2 (T\wedge \tau )+f(T\wedge \tau),
$$
with 
$$
f(t) = K^*_t + \int_0^t \frac{\lambda}{\gamma_A}\left(\exp(-\gamma_A K^*_s) - 1 \right) ds,  t \in [0,T].
$$
Let us define
$$g(t) = \frac{c_1+c_2}{c_1}\exp\left(c_1 \frac{\gamma_A}{\gamma_P+\gamma_A}(T-t)\right) - \frac{c_2}{c_1}.$$
We have
$ g'(t) =  -(c_1+c_2)\frac{\gamma_A}{\gamma_P + \gamma_A}\exp\left(c_1 \frac{\gamma_A}{\gamma_P+\gamma_A}(T-t)\right) .$

Therefore, 
\begin{align*}
\frac{\partial }{\partial t} K^*_t &= \frac{1}{\gamma_A} \frac{g'(t)}{g(t)}\\
&= \frac{1}{\gamma_P + \gamma_A} \left\{ \dfrac{-(c_1+c_2)\exp\left(c_1 \frac{\gamma_A}{\gamma_P+\gamma_A}(T-t)\right) }{ \frac{c_1+c_2}{c_1}\exp\left(c_1 \frac{\gamma_A}{\gamma_P+\gamma_A}(T-t)\right) - \frac{c_2}{c_1}} \right\} \\
&= \frac{c_1}{\gamma_P + \gamma_A} \left\{ \dfrac{-(c_1+c_2)\exp\left(c_1 \frac{\gamma_A}{\gamma_P+\gamma_A}(T-t)\right) }{  (c_1+c_2)\exp\left(c_1 \frac{\gamma_A}{\gamma_P+\gamma_A}(T-t)\right) - c_2} \right\}.\\
\end{align*}
Also
$$ \frac{\partial}{\partial t}\int_0^t \frac{\lambda}{\gamma_A}\left(\exp(-\gamma_A K^*_s) - 1 \right) ds =  \frac{\lambda}{\gamma_A}\left(\exp(-\gamma_A K^*_t) - 1 \right)$$
and so : 

\begin{align*}
 \frac{\partial}{\partial t}\int_0^t \frac{\lambda}{\gamma_A}\left(\exp(-\gamma_A K^*_s) - 1 \right) ds
 &=  \frac{\lambda}{\gamma_A}\left(\exp(-\gamma_A K^*_t) - 1 \right) \\
  &=  \frac{\lambda}{\gamma_A}\left( \frac{1}{g(t)} - 1 \right)  \quad \text{well-defined as $g(t) > 0$ on $[0,T]$}\\
&=  \frac{\lambda}{\gamma_A} \left\{ \frac{1}{ \frac{c_1+c_2}{c_1}\exp\left(c_1 \frac{\gamma_A}{\gamma_P+\gamma_A}(T-t)\right) - \frac{c_2}{c_1}} - 1\right\}\\
 &=  \frac{c_1 \lambda}{\gamma_A} \left\{ \frac{1}{ (c_1 + c_2)\exp\left(c_1 \frac{\gamma_A}{\gamma_P+\gamma_A}(T-t)\right) - c_2 } \right\} - \frac{\lambda}{\gamma_A}
\end{align*}
 
 Finally, we have : 
 \begin{align*}
  f'(t) &= \frac{\partial }{\partial t} K^*_t +  \frac{\partial}{\partial t}\int_0^t \frac{\lambda}{\gamma_A}\left(\exp(-\gamma_A K^*_s) - 1 \right) ds\\
  &=  \frac{c_1}{\gamma_P + \gamma_A} \left\{ \dfrac{-(c_1+c_2)\exp\left(c_1 \frac{\gamma_A}{\gamma_P+\gamma_A}(T-t)\right) }{  (c_1+c_2)\exp\left(c_1 \frac{\gamma_A}{\gamma_P+\gamma_A}(T-t)\right) - c_2} \right\} \\
  &+  \frac{c_1 \lambda}{\gamma_A} \left\{ \frac{1}{ (c_1 + c_2)\exp\left(c_1 \frac{\gamma_A}{\gamma_P+\gamma_A}(T-t)\right) - c_2 } \right\} - \frac{\lambda}{\gamma_A}\\
  &=   \frac{-c_1}{\gamma_P + \gamma_A}  \frac{1}{ (c_1 + c_2)\exp\left(c_1 \frac{\gamma_A}{\gamma_P+\gamma_A}(T-t)\right) - c_2 }\left\{ (c_1 + c_2)\exp\left(c_1 \frac{\gamma_A}{\gamma_P+\gamma_A}(T-t)\right) - \lambda \frac{\gamma_A}{\gamma_P + \gamma_A}\right\} \\
  & -\frac{\lambda}{\gamma_A} \\
  & = -\frac{c_1}{\gamma_P + \gamma_A }- \frac{\lambda}{\gamma_A}\quad \text{as $c_2 = \frac{\lambda \gamma_A}{\gamma_P + \gamma_A}$}\\
  & =  -\left( \frac{c_1}{\gamma_P + \gamma_A} + \frac{\lambda}{\gamma_A}\right).
 \end{align*}
 \end{proof}
 
 We may note that the default related part of the wage paid under full-Risk-Sharing also writes under this form, where we take the relevant values for $K_0^*, c_1$ and $c_2$. 
 \begin{Remark}
A little algebra gives
$$ f(t) = K^*_0  -\left( \frac{c_1}{\gamma_P + \gamma_A} + \frac{\lambda}{\gamma_A}\right)t = K^*_0  - \frac{\left( c_1 + c_2 \right)t}{\gamma_P + \gamma_A} $$
The slope of $f$ is of opposite sign to the sign of $c_1+c_2$. As a consequence, 
\begin{itemize}
\item If $c_1+c_2=0$ then $f(t)=0$ for every $t \in [0,T]$,
\item if $c_1+c_2>0$ then $f(0)=K^*_0>0$ for $t \le \tau$ and the function $t \to f(t)$ decreases,
\item if $c_1+c_2<0$ then $f(0)=K^*_0 <0$ for $t \le \tau$ and the function $t \to f(t)$ increases.
\end{itemize}
\end{Remark}

$K_0^*>0$ is the extra compensation asked at the signature of the contract by an agent who is more sensitive to the shutdown risk than the principal.
 We can easily visualize the sign of the control $K^*_0$ as a function of $\gamma_P,$ $\gamma_A$ and $\kappa$. Below we plot the sign  depending on the risk-aversions and fixing $\kappa=1$ and $\kappa=2$. The first two plots (Figure \ref{fig:test1} and \ref{fig:test2}) correspond to the full Risk-Sharing case whilst Figures \ref{fig:test3} and \ref{fig:test4} show the sign of $K^*$ in the Moral Hazard case. The $x$-axis holds the values of $\gamma_A$ and the $y$-axis the values of $\gamma_P$. Both risk-aversion constants are valued between $0$ and $10$ and the origin is in the bottom left corner. Blue encodes a negative sign and green a positive sign.

\begin{figure}[h!]
\centering
\begin{minipage}{0.33\textwidth}
  \centering
  \includegraphics[width=.70\linewidth]{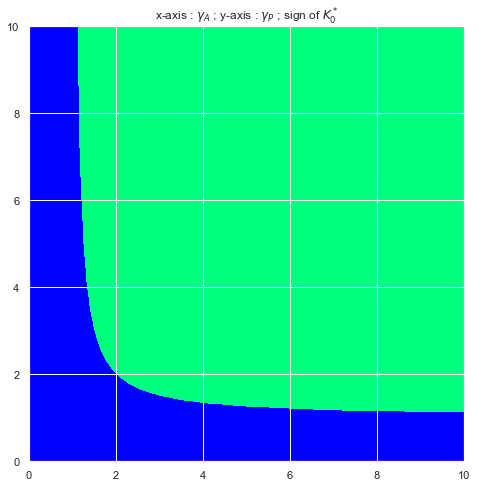}
  \captionof{figure}{Sign of $K_0^*$ depending on $\gamma_P$ and $\gamma_A$ for $\kappa=1$.}
  \label{fig:test1}
\end{minipage}%
~
~
\begin{minipage}{0.33\textwidth}
  \centering
  \includegraphics[width=.70\linewidth]{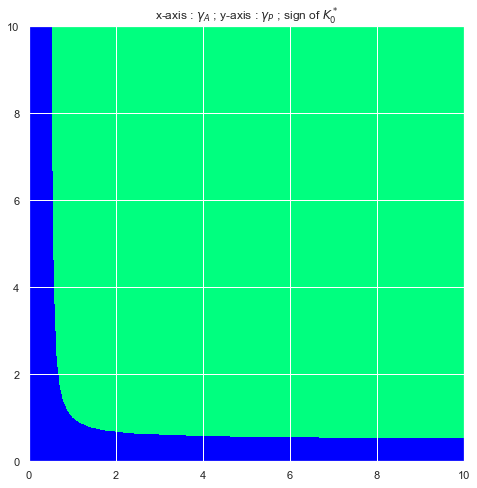}
  \captionof{figure}{Sign of $K_0^*$ depending on $\gamma_P$ and $\gamma_A$ for $\kappa=2$.}  \label{fig:test2}
\end{minipage}
\end{figure}

\begin{figure}[htb!]
\centering
\begin{minipage}{0.33\textwidth}
  \centering
  \includegraphics[width=.70\linewidth]{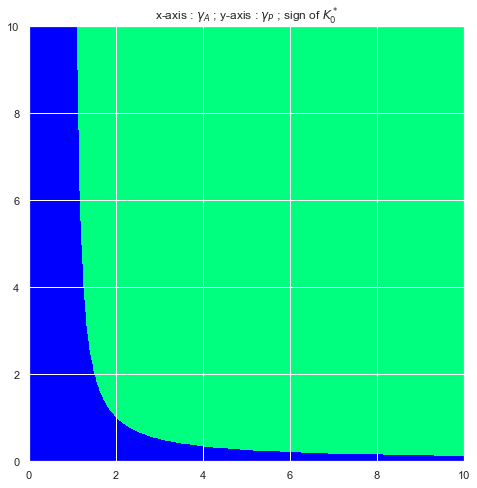}
  \captionof{figure}{Sign of $K_0^*$ depending on $\gamma_P$ and $\gamma_A$ for $\kappa=1$. }
  \label{fig:test3}
\end{minipage}%
~
~
\begin{minipage}{0.33\textwidth}
  \centering
  \includegraphics[width=.70\linewidth]{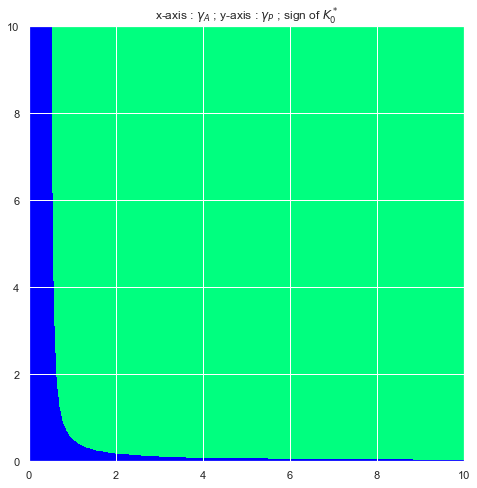}
  \captionof{figure}{Sign of $K_0^*$ depending on $\gamma_P$ and $\gamma_A$ for $\kappa=2$.} \label{fig:test4}
\end{minipage}
\end{figure}

We observe that in most situations, the sign of $K^*_0$ is positive. A negative sign occurs when either the principal or the agent are close to being risk-neutral (symmetrically so in the full Risk-Sharing case but asymmetrically so in the Moral Hazard case : the sign switches from negative to positive at a much lower level of risk-aversion for the principal than the agent). Also note that increasing the agent's effort coefficient $\kappa$ decreases the level of risk-aversion for which $K^*_0$ goes from positive to negative. 
Note that it is known by both the Principal and the Agent at time $0$ whether the contract will fall into either regime.\\

Our key result shows that the risk of shutdown adds an extra linear term to the optimal compensation that, up to the underlying constants, has the same structure under both full Risk Sharing and under Moral Hazard.
Even if the optimal contract separates the role of incentives from that of shutdown risk sharing, the amount of the {\it insurance deposit} $K^*_0$ depends strongly on the magnitude of the moral hazard problem. Therefore, we may naturally quantify the effect of Moral Hazard on the risk-sharing part. With this question in mind we compare the values of $K^*_0$ under Risk-Sharing and Moral Hazard for different parameter values.  This may be observed in the Figures \ref{KO1} to \ref{KO5} below where we represent the values of $K^*_0$ under Moral Hazard (red) and Risk-Sharing (blue) as a function of one of the underlying parameters ($\gamma_P, \gamma_A, \lambda, \kappa, T$) whilst fixing the remaining 4. \\

These figures lead to a crucial observation : under any given set of parameters, the value of $K^*_0$ under Moral Hazard is always greater than the value of $K^*_0$ under Risk Sharing. Note that in some settings, given some parameters the signs of the two cases are not always the same as the sign is that of $c_1+c_2$ which depends on the values of $\gamma_P$, $\gamma_A$ and $\kappa$ in a different manner in Risk-Sharing and Moral Hazard. Of course though, as soon as $K_0^*$ for Risk-Sharing is positive,  $K_0^*$ for Moral Hazard is positive too. We may note that although $\lambda$ and $T$ do not impact the sign, variations in their values have an impact on the magnitude $|K_0^*|$ : the higher the risk of early termination of the contract, the lower the amount of insurance requested by the agent and the longer the duration of the contract, the higher the amount of insurance requested by the agent.
\begin{figure}[htp]
\caption{Values of $K_0^*$ depending on $\gamma_P$}
\label{KO1}
\centering
\begin{minipage}{0.31\textwidth}
  \centering
  \includegraphics[width=.9\linewidth]{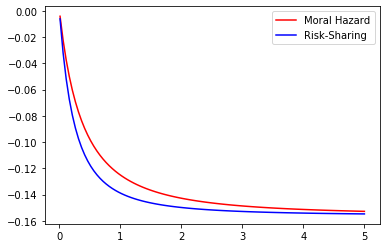}
   \par
  \footnotesize{$\quad \gamma_A=0.5, \lambda = 1, \kappa=1, T=1$}
\end{minipage}
~
~
\begin{minipage}{0.31\textwidth}
  \centering
  \includegraphics[width=.9\linewidth]{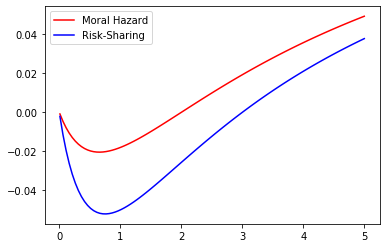} \par
  \footnotesize{$\quad \gamma_A=1.5, \lambda = 1, \kappa=1, T=1$}
  \label{fig:test1}
\end{minipage}%
~
~
\begin{minipage}{0.31\textwidth}
  \centering
  \includegraphics[width=.9\linewidth]{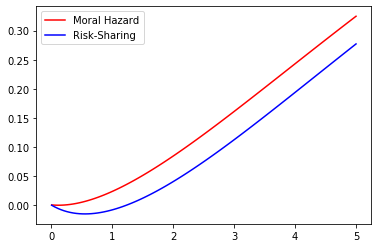}
  \footnotesize{$\quad \gamma_A=5, \lambda = 1, \kappa=1, T=1$}
\end{minipage}
\end{figure}

\begin{figure}[htp]
\caption{Values of $K_0^*$ depending on $\gamma_A$}
\label{KO2}
\centering
\begin{minipage}{0.31\textwidth}
  \centering
  \includegraphics[width=.9\linewidth]{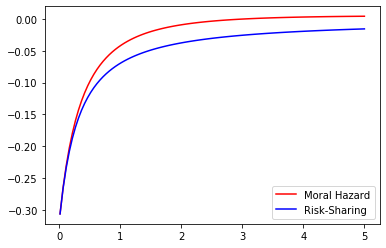}
   \par
  \footnotesize{$\quad \gamma_P=0.5, \lambda = 1, \kappa=1, T=1$}
\end{minipage}
~
~
\begin{minipage}{0.31\textwidth}
  \centering
  \includegraphics[width=.9\linewidth]{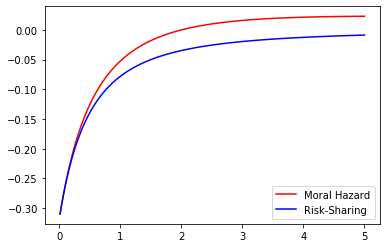} \par
  \footnotesize{$\quad \gamma_P=1, \lambda = 1, \kappa=1, T=1$}
  \label{fig:test1}
\end{minipage}%
~
~
\begin{minipage}{0.31\textwidth}
  \centering
  \includegraphics[width=.9\linewidth]{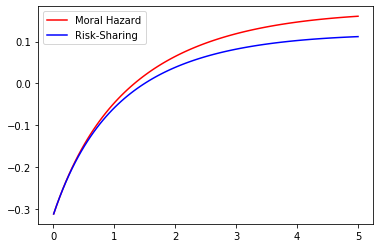}
  \footnotesize{$\quad \gamma_P=3, \lambda = 1, \kappa=1, T=1$}
\end{minipage}
\end{figure}

\begin{figure}[htp]
\caption{Values of $K_0^*$ depending on $\lambda$}
\label{KO3}
\centering
\begin{minipage}{0.31\textwidth}
  \centering
  \includegraphics[width=.9\linewidth]{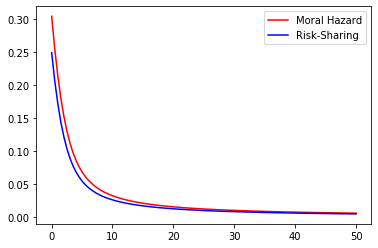} \par
  \footnotesize{$\quad \gamma_P=4, \gamma_A = 4, \kappa=1, T=1$}
  \label{fig:test1}
\end{minipage}%
~
~
\begin{minipage}{0.31\textwidth}
  \centering
  \includegraphics[width=.9\linewidth]{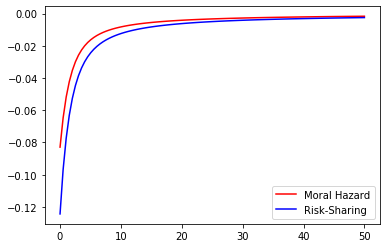}
   \par
  \footnotesize{$\quad \gamma_P=1, \gamma_A = 1, \kappa=1, T=1$}
\end{minipage}

\end{figure}

\begin{figure}[htp]
\caption{Values of $K_0^*$ depending on $\kappa$}
\label{KO4}
\centering
\begin{minipage}{0.31\textwidth}
  \centering
  \includegraphics[width=.9\linewidth]{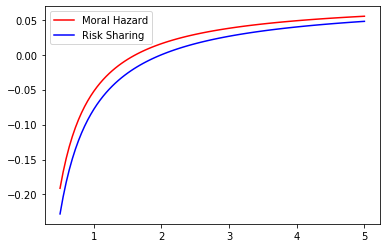} \par
  \footnotesize{$\quad \gamma_P=1, \gamma_A = 1, \lambda=1, T=1$}
  \label{fig:test1}
\end{minipage}%
~
~
\begin{minipage}{0.31\textwidth}
  \centering
  \includegraphics[width=.9\linewidth]{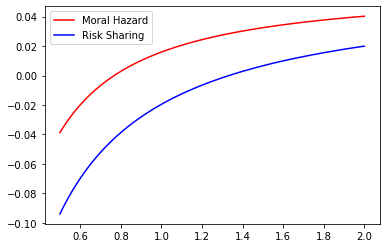}
   \par
  \footnotesize{$\quad \gamma_P=1, \gamma_A = 3, \lambda=1, T=1$}
\end{minipage}
~
~
\begin{minipage}{0.31\textwidth}
  \centering
  \includegraphics[width=.9\linewidth]{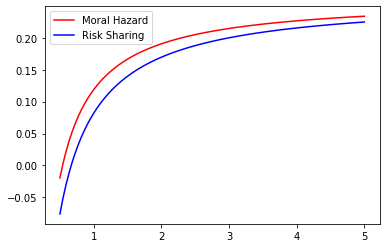}
   \par
  \footnotesize{$\quad \gamma_P=3, \gamma_A = 3, \lambda=1, T=1$}
\end{minipage}

\end{figure}

\begin{figure}[htp]
\caption{Values of $K_0^*$ depending on $T$}
\label{KO5}
\centering
\begin{minipage}{0.31\textwidth}
  \centering
  \includegraphics[width=.9\linewidth]{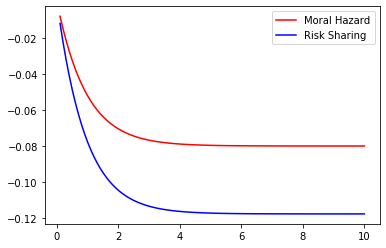} \par
  \footnotesize{$\quad \gamma_P=1, \gamma_A = 1, \lambda=1, \kappa=1$}
  \label{fig:test1}
\end{minipage}%
~
~
\begin{minipage}{0.31\textwidth}
  \centering
  \includegraphics[width=.9\linewidth]{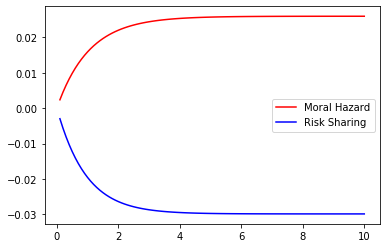}
   \par
  \footnotesize{$\quad \gamma_P=1, \gamma_A = 3, \lambda=1, \kappa=1$}
\end{minipage}
~
~
\begin{minipage}{0.31\textwidth}
  \centering
  \includegraphics[width=.9\linewidth]{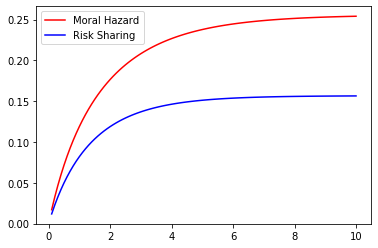}
   \par
  \footnotesize{$\quad \gamma_P=3, \gamma_A = 3, \lambda=1, \kappa=1$}
\end{minipage}

\end{figure}
\ \\

Additionally, the parameter $\lambda$ does have a quantifiable effect of the wage as it affects its expected value : 
$$
\E(f(T\wedge \tau))= K_0^*-\frac{c_1+c_2}{\gamma_P+\gamma_A}\left(\frac{1-e^{-\lambda T}}{T}\right).
$$

This expected value is represented below as a function of $\gamma_A$ and $\gamma_P$ for different values of $\lambda$ (with $T=1$ and $\kappa=1$ fixed). The first three figures (Figures \ref{fig:test5}, \ref{fig:test6} and \ref{fig:test7}  ) concern the full Risk-Sharing case whilst the second set of figures  (Figures \ref{fig:test8}, \ref{fig:test9} and \ref{fig:test10}) concern the Moral Hazard setting.  \\
\begin{figure}[htb!]
\centering

\begin{minipage}{0.31\textwidth}
  \centering
  \includegraphics[width=.95\linewidth]{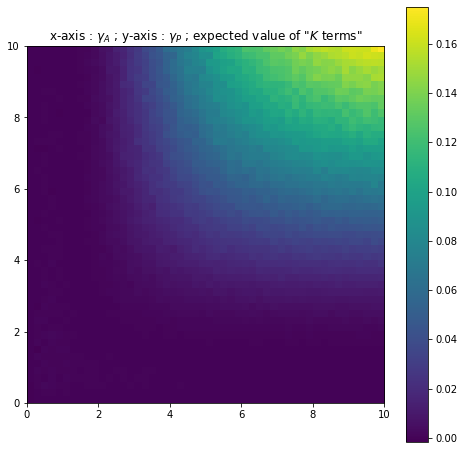}
  \captionof{figure}{Expected value depending on $\gamma_P$ and $\gamma_A$ for $\lambda=0.5$. }
  \label{fig:test5}
\end{minipage}%
~
~
\begin{minipage}{0.31\textwidth}
  \centering
  \includegraphics[width=.95\linewidth]{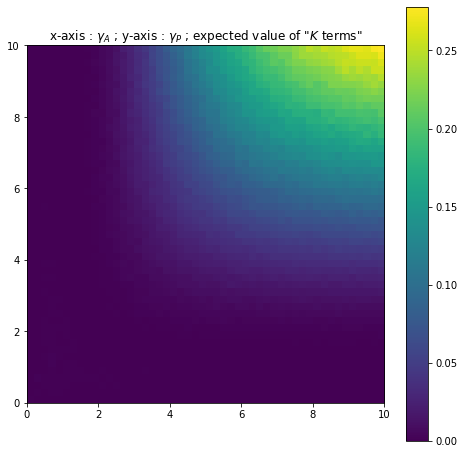}
  \captionof{figure}{Expected value depending on $\gamma_P$ and $\gamma_A$ for $\lambda=1$. }
  \label{fig:test6}
\end{minipage}%
~
~
\begin{minipage}{0.31\textwidth}
  \centering
  \includegraphics[width=.95\linewidth]{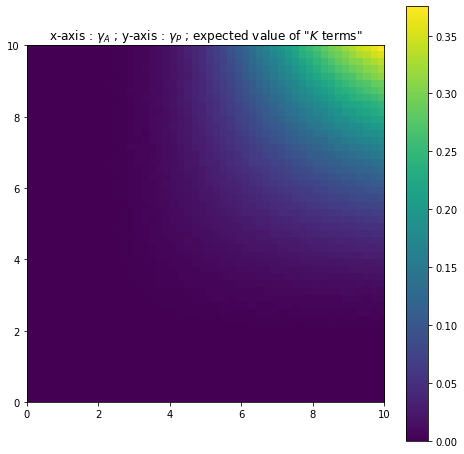}
  \captionof{figure}{Expected value depending on $\gamma_P$ and $\gamma_A$ for $\lambda=5$. }
  \label{fig:test7}
\end{minipage}
\end{figure}

\begin{figure}[htb!]
\centering

\begin{minipage}{0.31\textwidth}
  \centering
  \includegraphics[width=.95\linewidth]{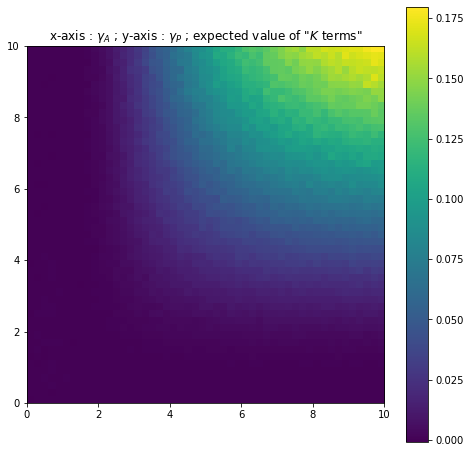}
  \captionof{figure}{Expected value depending on $\gamma_P$ and $\gamma_A$ for $\lambda=0.5$. }
  \label{fig:test8}
\end{minipage}%
~
~
\begin{minipage}{0.31\textwidth}
  \centering
  \includegraphics[width=.95\linewidth]{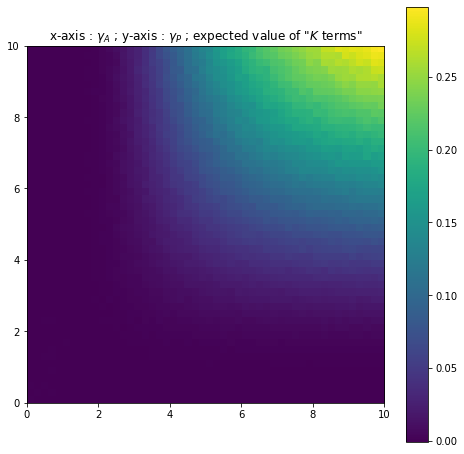}
  \captionof{figure}{Expected value depending on $\gamma_P$ and $\gamma_A$ for $\lambda=1$. }
  \label{fig:test9}
\end{minipage}%
~
~
\begin{minipage}{0.31\textwidth}
  \centering
  \includegraphics[width=.95\linewidth]{lam5K1RS}
  \captionof{figure}{Expected value depending on $\gamma_P$ and $\gamma_A$ for $\lambda=5$. }
  \label{fig:test10}
\end{minipage}
\end{figure}

The value seems to be in many cases very close to 0. As such, the agent earns, on average, a very similar wage to a "stopped" Holmstrom-Milgrom wage. However when the principal and the agent are particularly risk-averse, the expected value increases quite notably and the agent gains slightly more on average. This is in line with the papers by Hoffman and Pfeil \cite{HoffmanPfeil} and Bertrand and Mullainathan \cite{Bertrand} which show the agent must be rewarded for a risk that is beyond his control. Note that the simulations show that for fixed levels of risk-aversion the expected value increases as $\lambda$ increases. For example for $\gamma_P=\gamma_A=7$, in Figure \ref{fig:test8} the expected value is approximately worth $0.075$ and in Figure \ref{fig:test9} it is approximately equal to $0.15$. As such the risk-averse agent is increasingly rewarded for the uncontrollable risk.  

Finally, we can make a few further comments related to the underlying expected utilities in this new contracting setting. First, the agent's expected utility is the same under both full Risk-Sharing and Moral Hazard. Indeed he walks away with his participation constraint : 
$$ \E\left[ U_A\left(W^* - \int_0^T \kappa(a^*_s) ds\right) \right] = U_A(y_{PC}),$$
where $(W^*, a^*)$ designates the Risk-Sharing or Moral Hazard optimal contract under shutdown risk. Such a result also holds in the same setting without shutdown risk : the agent tracts average the same expected utility under full Risk-Sharing or Moral Hazard, with or without an underlying agency-free external risk. \\
When it comes to the principal we may wonder how his expected utility may be affected by the possibility of a halt in production. In particular we may question what the principal loses in not being able to observe the agent's action under a likelihood of agency-free external risk ? To answer this we denote as $V_0^{RS}(0,x,y)$ the Principal's expected utility under full Risk-Sharing and $V_0^{MH}(0,x,y)$ under Moral Hazard. We have that : 
$$V_0^{RS}(0,x,y) = U_P(x-y) \Phi_0^{RS}(0) \quad \text{and} \quad V_0^{MH}(0,x,y) = U_P(x-y) \Phi_0^{MH}(0)$$
where $\Phi_0^{RS}$ and $\Phi_0^{MH}$ are the related functions from Theorem \ref{theo:mainFB} and Theorem \ref{theo:main1}. With these expressions we have that : 
$$ \dfrac{ V_0^{MH}(0,x,y) }{V_0^{RS}(0,x,y) } = \frac{ \Phi_0^{MH}(0)}{ \Phi_0^{RS}(0)}.$$
The Moral Hazard problem involves optimizing across a more restricted set of contracts. Therefore we know that : 
$$V_0^{MH}(0,x,y) \leq V_0^{RS}(0,x,y),$$
and as $U_P(x-y)<0$ : 
$$\Phi_0^{RS}(0) \leq \Phi_0^{MH}(0).$$ 
We may question whether this inequality leads to a big gap between the expected utilities and we answer this by plotting the ratio $\frac{\Phi_0^{MH}(0)}{\Phi_0^{RS}(0)}$ for different values of $\lambda$ and with $T=1$ and $\kappa=1$ fixed.

\begin{figure}[htb!]
\centering

\begin{minipage}{0.31\textwidth}
  \centering
  \includegraphics[width=.95\linewidth]{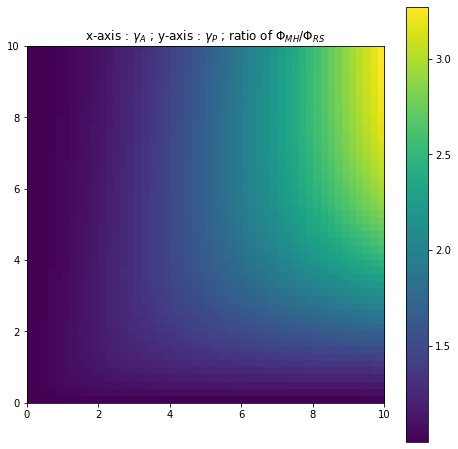}
  \captionof{figure}{ Ratio depending on $\gamma_P$ and $\gamma_A$ for $\lambda=0.5$. }
  \label{fig:test11}
\end{minipage}%
~
~
\begin{minipage}{0.31\textwidth}
  \centering
  \includegraphics[width=.95\linewidth]{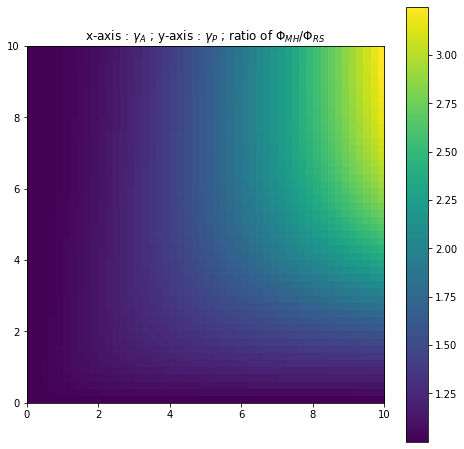}
  \captionof{figure}{Ratio depending on  $\gamma_P$ and $\gamma_A$ for $\lambda=1$. }
  \label{fig:test12}
\end{minipage}%
~
~
\begin{minipage}{0.31\textwidth}
  \centering
  \includegraphics[width=.95\linewidth]{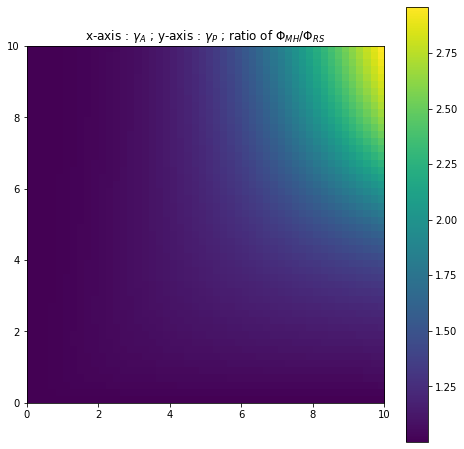}
  \captionof{figure}{Ratio depending on $\gamma_P$ and $\gamma_A$ for $\lambda=5$. }
  \label{fig:test13}
\end{minipage}
\end{figure}

We first observe a standard result : for low levels of risk-aversion the ratio is close to 1 and the principal does not lose much by not observing the agent's actions. As the values of risk-aversion increase the principal loses out more and more by not being in a first best setting. This classical result comes with an observation that is specific to the presence of shutdown risk : as $\lambda$ increases (and therefore as the chance of shutdown risk occurring before $T$ increases), the ratio stays close to 1 for higher and higher levels of risk-aversion. For example in Figure \ref{fig:test11} when $\gamma_P=\gamma_A=6$ we observe that $ \frac{\Phi_0^{MH}(0)}{\Phi_0^{RS}(0)} \approx 2$ yet in Figure \ref{fig:test13} for the same levels of risk-aversion we have that  $ \frac{\Phi_0^{MH}(0)}{\Phi_0^{RS}(0)} \approx 1.$ So a high possibility of some production halt occurring reduces the gap between the full Risk-Sharing contract and the Moral Hazard contract. Such a phenomena may be due to the fact that a high possibility of a production halt at some point in the time interval means that the wage process evolves on a time period that is on average shorter before stopping. There is thus less time for a significant gap to appear between the full Risk-Sharing case and the Moral Hazard case. \\

As this analysis comes to a close we finish this section by discussing a possibility for extension with more general deterministic compensators.

\subsection{General deterministic compensators $(\Lambda_t)_{t \in [0,T]}$} \label{deterministicintensity}

Throughout this paper, we have considered a constant compensator $\lambda$ for the jump process. This choice allows for clearer calculations but it is key to note that our results extend to the case where $\lambda$ is no longer constant such as : 

$$ \Lambda_t = \int_0^t \lambda_s ds,$$
with $(\lambda_s)_{s \in [0,T]}$ some deterministic positive mapping such that $\Lambda_T<+\infty$. The proofs for the optimal contracting are simply a direct extension of the proofs of the previous sections. Of course due to the independence between $B$ and $N$, a time dependent compensator does not induce any change to the Holmstrom-Milgrom part of the wages. Only the part related to $K^*$ is affected. We provide the details in the following. \\

\noindent \textbf{\scshape{The full Risk-Sharing problem}}\\

\noindent The optimal wage for the Risk-Sharing problem in such a setting is of the form~: 
\begin{align*}
W_t &= y^* + \int_0^t Z^*_s (1-N_s)dB_s + \int_0^t K^*_s (1-N_s) dM_s \\&+ \int_0^t \left\{ \frac{\gamma_A}{2} {Z_s^*}^2 + \kappa(a^*_s(1-N_s)) + \frac{\lambda_s}{\gamma_A} [\exp(-\gamma_A K^*_s) - 1 + \gamma_A K^*_s ]\right\} (1-N_s)ds,
\end{align*}
where : 
$$ y^* = y_{PC}, a^*_t = \frac{1}{\kappa}, Z^*_t = \frac{\gamma_P}{\gamma_P + \gamma_A} \quad \text{and} \quad K^*_t = \frac{1}{\gamma_P + \gamma_A} \log(\Phi_0(t)),$$
with $\Phi_0(t)$ solution to the Bernouilli equation : 

$$ \Phi'_0(t) + c_1(t)\Phi_0(t) + c_2(t) \Phi_0(t)^{\frac{\gamma_P}{\gamma_P+\gamma_A}}=0, \quad \Phi_0(T) = 1,$$
where 
$$ c_1(t) = \frac{\gamma_P^2 \gamma_A}{2(\gamma_P+\gamma_A) } - \frac{\gamma_P}{2\kappa} - \lambda_t \frac{\gamma_P + \gamma_A}{\gamma_A} \quad \text{and} \quad c_2(t) = \lambda_t \frac{\gamma_P + \gamma_A}{\gamma_A}.$$

\noindent\textbf{\scshape{The Moral Hazard problem}}\\

\noindent The optimal wage in the Moral-Hazard problem is again of the form : 
$$W_t = y^* + \int_0^t Z_s^*(1-N_s) dB^*_s + \int_0^t K_s^* (1-N_s)dM_s + \int_0^t \left\{ \frac{1}{2} \gamma_A  {Z_s^*}^2 - \frac{{Z_s^*}^2}{2 \kappa} + \lambda_s K_s^*  +  \frac{\lambda_s}{\gamma_A}(e^{-\gamma_A K_s^*}-1) \right\} (1-N_s)ds, $$

where : 
$$ y^* = y_{PC}, Z^*_t = \frac{\gamma_P + \kappa^{-1}}{\gamma_P + \gamma_A + \kappa^{-1}} \quad \text{and} \quad K^*_t = \frac{1}{\gamma_P + \gamma_A} \log(\Phi_0(t)) ,$$
with $\Phi_0(t)$ solution to the Bernouilli equation : 

$$ \Phi'_0(t) + c_1(t)\Phi_0(t) + c_2(t) \Phi_0(t)^{\frac{\gamma_P}{\gamma_P+\gamma_A}}=0, \quad \Phi_0(T) = 1,$$

with : $$  c_1(t) = \frac{\gamma_P^2 \gamma_A}{2(\gamma_P+\gamma_A + \kappa^{-1}) } - \frac{\gamma_P\kappa^{-1}(\gamma_P+\kappa^{-1})}{2{(\gamma_P + \gamma_A + \kappa^{-1})}} - \lambda_t \frac{\gamma_P + \gamma_A}{\gamma_A} \quad \text{and} \quad c_2(t) = \lambda_t \frac{\gamma_P + \gamma_A}{\gamma_A}.$$

%Observe that for $t \in [0,T]$, we have $c_1(t)+c_2(t)=c_1+c_2.$

%\begin{Remark}
%Lemma \ref{lemma:K} still holds in this context : 
%$$ K_t^* =  \frac{1}{\gamma_A} \log\left(\E\left[\exp\left(\frac{\gamma_A}{\gamma_P+\gamma_A}(c_1(t)+c_2(t))((T-t) \wedge \tau)\right)\right]\right)\quad t \in [0,T]. $$
%The sign of $K^*$ is thus once again constant across $[0,T]$ and given by the criteria of Lemma \ref{lemma:signK}. 
%\end{Remark}

\section{Mitigating the effects of agency-free external risk}
\label{sec:M}

This paper has so far modeled the occurence of a halt as a complete fatality suffered by both parties in the contracting agreement. Yet the recent crisis has highlighted the ability of humans and businesses to react and adapt when faced with adversity. We now include such phenomena in the contracting setting by allowing the principal to invest upon a halt in order to continue some form of (possibly disrupted) production. This is quite a natural and realistic variant on our initial model. Indeed when faced with a period of lockdown, companies may for example invest in teleworking infrastructure so that a number of employees whose jobs are doable remotely can continue to work. Similarly, jobs that require some form of presence could continue if companies invest in protective equipment and adapt their organization. We may wonder how such a mechanism may affect optimal contracting. 

\subsection{Setting for mitigation}
Mathematically, we consider that the production process evolves as previously up until $\tau \wedge T$. If a halt happens at some time $\tau \leq T$, we allow the principal to invest an amount $i>0$ to continue production at a degraded level $\theta \in (0,1)$. It is assumed that the investment decision is at the principal's convenience. It is modeled by a control $D$ which is a ${\mathbb G}_\tau$-measurable random variable with values in $\{0,1\}$. The $\theta$ parameter is firm-specific and  reflects the effectiveness of the post-shutdown reorganization.\\
Under the initial probability measure $\mathbb{P}^0$, the output process $X_t$ evolves as
$$
X_t=x_0+\int_0^t ((1-N_s)+N_s\ind_{D=1})\,dB_s.
$$
We recall from \cite{Jeulin:1980} Lemma 4.4. the decomposition of a $\mathbb G$-adapted process $\phi$. There exist a $\mathbb F$-adapted process $\phi^0$ and a family of processes $(\phi^1_t (u), u \le t \le T)$ that are $\mathbb F_t \otimes {\cal B}(\R_+)$ measurable such that
$$
\phi_t=\phi_t^0\ind_{t<\tau}+\phi^1_t(\tau)\ind_{t \ge \tau}.
$$ 
{\it Contract:}  A contract $W$ is a $\mathbb{G}_{T}$-measurable random variable satisfying $\E(\exp(-2\gamma_A W))<+\infty$ of the form
$$
W=W^0\ind_{T<\tau}+(W_T^{1,1}(\tau)\ind_{D=1}+W^{1,0}_T(\tau)\ind_{D=0})\ind_{\tau \le T}.
$$
where $W^0$ is $\mathbb F_T$-mesurable and $W_t^{1,1}(u)$ and $W_t^{1,0}(u)$ are $\mathbb F_t \otimes {\cal B}(\R_+)$ measurable. We will assume that $W_T^{1,0}(\tau)=W_{T\wedge \tau}^{1,0}(\tau)$ since in the absence of investment, it is no longer necessary to give incentives after $\tau$.\\

{\it Effort process:} In this setting, the agent will adapt his effort to the occurence of the shutdown risk. This is mathematically modeled by a $\mathbb G$-adapted process $(a_t)_t$ in the form
$$
a_t=a_t^0\ind_{t<\tau}+a^1_t(\tau)\ind_{t \ge \tau}
$$ 
where $a^0$ and $a^1$ are respectively $\mathbb F$-adapted and $\mathbb F_t \otimes {\cal B}(\R_+)$ measurable. Furthermore, we assume that the effort processes are  bounded by some constant $A$.
We then define $\P^a$ as $ \frac{d\mathbb{P}^a}{d\mathbb{P}^0}|\mathcal{G}_T =L_T^\theta$, with
$$
L_T^\theta = \exp\left( \int_0^T a_s^0(1-N_s)+\theta a^1_s(\tau)N_s\ind_{D=1} dB_s - \frac{1}{2} \int_0^T (a_s^0)^2 (1-N_s)+\theta^2(a^1_s(\tau))^2N_s\ind_{D=1}ds \right)$$
 Because the processes $a^0$ and $a^1$ are bounded, $(B^a_t)_{t \in [0,T]}$ with
 $$
 B^a_t = B_t - \int_0^t (a_s^0(1-N_s)+\theta a^1_s(\tau)N_s\ind_{D=1})\, ds, t \in [0,T]
 $$ 
 is a $\mathbb{G}$-Brownian motion under $\mathbb{P}^a$. Under $\P^a$, the output process evolves as
 $$
 X_t=x_0+\int_0^t (a_s^0(1-N_s)+\theta a^1_s(\tau)N_s\ind_{D=1})\, ds+\int_0^t ((1-N_s)+N_s\ind_{D=1})\,dB_s.
 $$

\subsection{The Optimal contract}

We first make the following observation. After $\tau$, if the default time occurs before the maturity of the contract, the principal has a binary decision to take. If she decides to not invest, she gets the value $V^{1,0}(x,y)=U_P(x-y)$ where $x$ 
is the level of input and $y$ is the agent continuation value.
On the other hand, if she decides to invest, she will face for $t \ge \tau$ the moral hazard problem of Holmstrom and Milgrom for which we know the optimal contract and the associated value function
$$
V^{1,1}(t,x,y)=U_P(x-y) \Phi_1(t, \theta) 
$$
where $\Phi_1(t, \theta) = \exp(-\gamma_P C_{inv}(T-t))$ and $C_{inv}:= \dfrac{\left(\gamma_P + \frac{\theta^2}{ \kappa}\right)^2}{2\left(\gamma_P + \gamma_A + \frac{\theta^2}{\kappa} \right)} - \frac{\gamma_P}{2}.$\\
Because the principal has to pay a sunk cost $i>0$ to invest, she will decide optimally to invest if and only if at  $\tau$ for a given $(x,y)$, she observes
$$
V^{1,1}(\tau,x,y) \ge V^{1,0}(x,y),
$$
or equivalently $C_{inv}(T-\tau) >i$. Hence, if $C_{inv}>0$, the optimal control will be $D^*=\ind_{\{\tau < T-\frac{i}{C_{inv}}\}}.$
To sum up, we have
\begin{Lemma}
\label{lemma:new1}
\begin{enumerate}
\item Investment for mitigation is \underline{never} optimal upon a halt if : 
$$ C_{inv} < 0 \quad \text{or} \quad i > T C_{inv}. $$
\item Now suppose that :
$$ C_{inv} > 0 \quad \text{and} \quad i < T C_{inv}. $$
Mitigation is optimal up until the cutoff time $t_{max}$ defined as : 
$$ t_{max} := T - \frac{i}{C_{inv}}. $$
Note that  $i < T C_{inv}$ guarantees that $t_{max} \geq 0.$
\end{enumerate}
\end{Lemma}

We are in a position to solve the before-default principal problem. Proceeding analogously as in Section \ref{sec:OC}, the before-default value function is given by the Markovian control problem
$$
V_P=\sup_{y \ge y_{PC}}V_0(0,x,y),
$$
with
\begin{equation}
V(0,x_0,y_0)=\sup_{\pi=(Z,K) \in \zeta} \E\left[ U_P(X_{T}^{\pi} - W_{T}^{\pi} )(1-N_T)+\int_0^T \max(V^{1,0}(X_t^\pi,W_t^{\pi}),V^{1,1}(t,X_t^\pi,W_t^{\pi}))\lambda e^{-\lambda t}\,dt \right],
\end{equation}
and 
$$
dX_t=a^*(Z_t)(1-N_t)\,dt+(1-N_t)\,dB_t^*,
$$
\begin{align}\label{wageinvest}
dW^{\pi}_t &=  Z_s(1-N_s) dB_s^* + K_s(1-N_s) dM_s \\
&+ \left\{ \frac{\gamma_A}{2}   Z_s^2 + \kappa(a^*(Z_s))+ \frac{\lambda}{\gamma_A} [\exp(-\gamma_A K_s) - 1 + \gamma_A K_s ] \right\} (1-N_s)ds.\nonumber
\end{align}
\begin{Remark} To be perfectly complete, we develop in the appendix the martingale optimality principle which makes it possible to obtain the dynamics \eqref{wageinvest}.
\end{Remark}

\begin{Theorem}
\label{theo:main2}
We have the following explicit characterizations of the optimal contracts. Assume the constant $A$ in the definition of the set of admissible efforts ${\cal B}$ satisfies 
$$
A> \dfrac{\gamma_P + \kappa^{-1}}{\kappa(\gamma_P + \gamma_A) + 1}.
$$
Let $Z^*_t  = \dfrac{\gamma_P + \kappa^{-1}}{\gamma_P + \gamma_A + \kappa^{-1}}$ and 

 $$K^*_t =\frac{1}{\gamma_P+\gamma_A}\log\left(\frac{\Phi_0(t)}{\min \Big\{1, \exp(\gamma_Pi) \Phi_1(t, \theta)\Big\}}\right)$$

with $\Phi_0$ as defined above with :
$$  c_1 := \frac{\gamma_P^2 \gamma_A}{2(\gamma_P+\gamma_A + \kappa^{-1}) } - \frac{\gamma_P \kappa^{-1} (\gamma_P+\kappa^{-1})}{2{(\gamma_P + \gamma_A + \kappa^{-1})}} - \lambda \frac{\gamma_P + \gamma_A}{\gamma_A} $$
 \text{and} 
 $$ c_2(t) = \lambda \frac{\gamma_P + \gamma_A}{\gamma_A} \min \Big\{1, \exp(\gamma_Pi) \Phi_1(t, \theta)\Big\}^{\frac{\gamma_A}{\gamma_P+\gamma_A}}$$

Then $(y_{PC},Z^*,K^*)$ parametrizes the optimal wage for the Moral Hazard problem with a possibility for mitigation. The Agent performs the optimal action $\dfrac{Z^*}{\kappa}$ before $\tau$ and $\dfrac{\theta Z^*}{\kappa}$ after $\tau$ 
when $\tau < T-i/C_{inv}$. 
\end{Theorem}
\begin{proof}

The reasoning used to compute the optimal Moral Hazard contract very much parallels the reasoning used above in Section \ref{sec:OC}. As a consequence, we are much more brief in the following proof.\\
The Hamilton-Jacobi-Bellman equation associated to the value function $V_0$ is the following : 
\begin{align}
\label{eq:HJB2}
0=\partial_t v_0(t,x,y) + \inf_{Z,K} \left\{ \partial_xv_0(t,x,y)\dfrac{Z}{\kappa} + \partial_yv_0(t,x,y) \left[\frac{\gamma_A}{2} Z^2 + \frac{Z^2}{2 \kappa}  + \frac{\lambda}{\gamma_A} [\exp(-\gamma_A K) - 1]\right] \right. \nonumber\\
\left.+ \lambda \left[ v_1(t,x,y+K) - v_0(t,x,y)\right]+ \partial_{yy} v_0(t,x,y) \frac{Z^2}{2} + \frac{1}{2} \partial_{xx} v_0(t,x,y)+ \partial_{xy}v_0(t,x,y) Z 
 \right\},
\end{align}
with the boundary condition :
$$v_0(T,x,y) = U_P(x-y)$$
where
$$v_1(t,x,y) =   U_P(x-y) \min \Big\{1, \exp(\gamma_Pi) \Phi_1(t, \theta)\Big\}. 
$$

\begin{Lemma}
\label{lem:temp10}
Assume $A> \dfrac{\gamma_P + \kappa^{-1}}{\gamma_P + \gamma_A + \kappa^{-1}}$. The function $v_0(t,x,y)=U_P(x-y) \Phi_0(t),$
with 
$$ \Phi_0(t) := \exp(-c_1 t) \left\{ \exp(c_1 \frac{\gamma_A}{\gamma_P + \gamma_A}T) + \frac{\gamma_A}{\gamma_P+\gamma_A} \int_t^T c_2(s) \exp(\frac{\gamma_A}{\gamma_P + \gamma_A} c_1 s) ds \right\}^{\frac{\gamma_P+\gamma_A}{A}}$$
where 
$$  c_1 = \frac{\gamma_P^2 \gamma_A}{2(\gamma_P+\gamma_A + \kappa^{-1}) } - \frac{\gamma_P \kappa^{-1}(\gamma_P+\kappa^{-1})}{2{(\gamma_P + \gamma_A + \kappa^{-1})}} - \lambda \frac{\gamma_P + \gamma_A}{\gamma_A} $$
 \text{and}
 $$c_2(t) = \lambda \frac{\gamma_P + \gamma_A}{\gamma_A} \min \Big\{1, \exp(\gamma_Pi) \Phi_1(t, \theta)\Big\}^{\frac{\gamma_A}{\gamma_P+\gamma_A}}.$$
solves in the classical sense the HJB equation (\ref{eq:HJB2}). In particular $Z^*_t = \dfrac{\gamma_P + \kappa^{-1}}{\gamma_P + \gamma_A + \kappa^{-1}}$ and $K^*_t =\frac{1}{\gamma_P+\gamma_A}\log\left(\dfrac{\Phi_0(t)}{\min \Big\{1, \exp(\gamma_Pi) \Phi_1(t, \theta)\Big\}}\right)$.
\end{Lemma}
\begin{proof}
The proof of this lemma is a direct adaptation of the proof of Lemma \ref{lem:temp2} to which we refer the reader.
\end{proof}

The proof of the final result relies on the regularity of $v_0$ and a standard verification result. Because the controls are free of $y$, we deduce that $V_P=V_0(0,x,y_{PC}).$
\end{proof}

The main change brought about by investment involves the halt related control $K^*$. Indeed the optimal  $Z^*$ in the Moral Hazard case are simply the optimal "Holmström-Milgrom" controls for the related production process.  At first glance, the optimal control $K^*$ seems to be quite different from that of Theorem \ref{theo:main1}. However one may verify that when we are in a setting where investment is never optimal (through the criteria of Lemma \ref{lemma:new1}), the expression for $K^*$ simplifies to exactly that of Theorem \ref{theo:main1}. The key to deduce this is that in such a setting, $\min \Big\{1, \exp(\gamma_Pi) \Phi_1(t, \theta)\Big\}=1$. We may therefore focus our analysis on the effects on investment when investing may be optimal (i.e. when $ C_{inv} > 0 \quad \text{and} \quad i < T C_{inv}$).  In such a setting, $K^*$ has two phases : 
\begin{itemize}
\item[-] before $t_{max}$, $K^*$ is adjusted to account for the possibility of risk mitigation
\item[-] after $t_{max},$ $K^*$ has the same values as without mitigation. Indeed : 
$$\min \Big\{1, \exp(\gamma_Pi) \Phi_1(t, \theta)\Big\}=1 \quad \text{for} \quad t \geq t_{max}.$$
\end{itemize}

We are able to analyze the effect of different parameters and to do so represent the deterministic part of $K^*$ as a function of time in the following figures. \\

 We fix parameters $\gamma_P = \kappa = T = 1, \gamma_A = 0.5$ : again this allows for mitigation to be optimal before some $t_{max}$. 
\begin{figure}[htb!]
\centering

\begin{minipage}{0.31\textwidth}
  \centering
  \includegraphics[width=.95\linewidth]{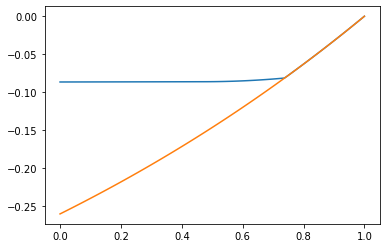}
  \captionof{figure}{$\lambda=0.5$, $i=0.1$, $\theta=0.9$ }
  \label{fig:test22}
\end{minipage}%
~
~
\begin{minipage}{0.31\textwidth}
  \centering
  \includegraphics[width=.95\linewidth]{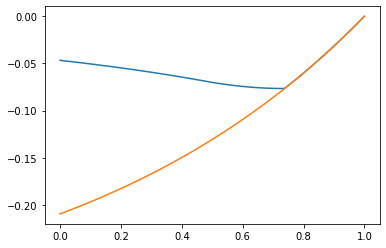}
  \captionof{figure}{$\lambda=1$, $i=0.1$, $\theta=0.9$. }
  \label{fig:test23}
\end{minipage}%
~
~
\begin{minipage}{0.31\textwidth}
  \centering
  \includegraphics[width=.95\linewidth]{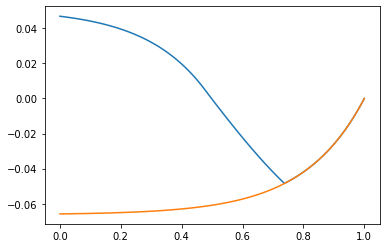}
  \captionof{figure}{$\lambda=5$, $i=0.1$, $\theta=0.9$. }
  \label{fig:test24}
\end{minipage}
\end{figure}

\begin{figure}[htb!]
\centering
\begin{minipage}{0.31\textwidth}
  \centering
  \includegraphics[width=.95\linewidth]{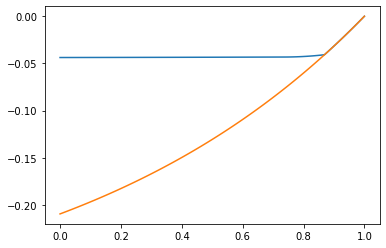}
  \captionof{figure}{$i=0.05$, $\lambda=1$,  $\theta=0.9$ }
  \label{fig:test25}
\end{minipage}%
~
~
\begin{minipage}{0.31\textwidth}
  \centering
  \includegraphics[width=.95\linewidth]{MH1}
  \captionof{figure}{$i=0.1$, $\lambda=1$,  $\theta=0.9$ . }
  \label{fig:test26}
\end{minipage}%
~
~
\begin{minipage}{0.31\textwidth}
  \centering
  \includegraphics[width=.95\linewidth]{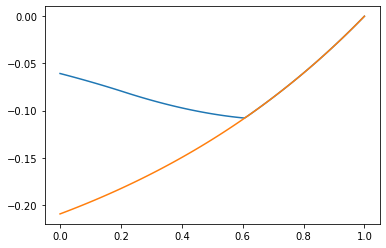}
  \captionof{figure}{$i=0.15$, $\lambda=1$,  $\theta=0.9$ . }
  \label{fig:test27}
\end{minipage}
\end{figure}

\begin{figure}[htb!]
\centering
\begin{minipage}{0.31\textwidth}
  \centering
  \includegraphics[width=.95\linewidth]{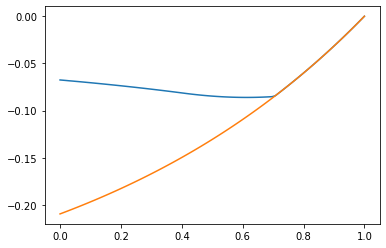}
  \captionof{figure}{$i=0.1$, $\lambda=1$,  $\theta=0.85$ }
  \label{fig:test28}
\end{minipage}%
~
~
\begin{minipage}{0.31\textwidth}
  \centering
  \includegraphics[width=.95\linewidth]{MH1}
  \captionof{figure}{$i=0.1$, $\lambda=1$,  $\theta=0.9$ . }
  \label{fig:test29}
\end{minipage}%
\end{figure}

We immediately observe that with mitigation, the value of $K^*$ before $t_{max}$ and is higher than without mitigation : the possibility for mitigation shrinks the opportunities for speculation  (see Figures \ref{fig:test22} to \ref{fig:test24}) and increasingly so as the probability of a halt increases. In fact the sign of $K^*$ may now change over the duration of the contracting period : see  Figure \ref{fig:test24}. Quite naturally, $t_{max}$ varies with $\theta$ and $i$. Indeed it decreases as $i$ increases or $\theta$ decreases : as the cost of investment increases and/or the level of degradation in continued production increases, more time is needed for investment for continued production to be worth it. \\

\newpage
\section{Appendix}
We sketch the martingale optimality principle arising from the Agent’s problem in the investment setting. We set
$\mathbb{H}^2_{\mathbb{G}}$ is the set of $\mathbb{G}-$ adapted processes $Z$ with $\E[\int_0^T Z_s^2 ds] < +\infty$ and $\mathbb{S}^2_{\mathbb{G}}$ is the set of $\mathbb{G}-$predictable processes $Y$ with cadlag paths such that $Z$ with $\E[\sup_{t \in [0,T]}Y^2_t] < +\infty$. For a   $\mathbb{G}_{\tau}-$measurable random variable $D$ with values in $\left\{ 0,1 \right\}$ that models the investment decision, we consider
a contract $W$ as a $\mathbb{G}_T$ measurable r.v. which can be decomposed under the form : 
$$ W = W^{0}\textbf{1}_{T < \tau} + \left( W^{1,1}_T(\tau)\textbf{1}_{D=1} + W^{1,0}_T(\tau)\textbf{1}_{D=0}\right)\textbf{1}_{\tau \leq T},$$
where
$W^{0}$ is $\mathbb{F}_T$-measurable, and $W_t^{1,1}(u)$ and $W_t^{1,0}(u)$ are $\mathbb{F}_t \otimes \mathcal{B}(\mathbb{R}_+)$ measurable, with in particular  $W^{1,0}_T(\tau) = W^{1,0}_{T \wedge \tau}(\tau). $
Given a contract $W$, the agent faces the following control problem,
$$ \sup_{a \in \mathcal{B}} \mathbb{E}^{\mathbb{P}^{a}}\E\left[ U_A\left(W - \int_0^T \kappa(a_s) ds\right)\right].$$
Remember that an effort process is now a $\mathbb{G}-$adapted process $(a_t)_{t \in [0,T]}$ and consequently has the form :
$$ a_t = a_t^{0}\textbf{1}_{t < \tau} + a_t^{1}(\tau) \textbf{1}_{t \geq \tau}$$
where $a^0$ is $\mathbb{F}-$adapted and $a^1$ is $\mathbb{F}_t \otimes \mathcal{B}(\mathbb{R}_+)$ measurable and where both are assumed bounded by some constant $A$. By convention, we still denote by $\mathcal{B}$ the set of such effort.

\begin{Lemma}
Suppose that there exists some unique triplet $(Y, Z, K)$ in $\mathbb{S}^{2}_{\mathbb{G}} \times \mathbb{H}^2_{\mathbb{G}} \times  \mathbb{H}^2_{\mathbb{G}} $ such that : 
$$ Y_t = W - \int_t^T Z_s((1-N_s) + N_s \textbf{1}_{D=1})dB_s - \int_t^T K_s(1-N_s) dM_s - \int_t^T f(s, Z_s, K_s) ds,$$
where 
\begin{align*}
 f(s, Z_s, K_s) &=  \left( \lambda K_s + \frac{\lambda}{\gamma_A} (e^{-\gamma_A K_s}-1)\right)(1-N_s) \\&+
 \frac{1}{2} \gamma_A Z_s^2 ((1-N_s) + N_s \textbf{1}_{D=1})
+ \inf_{a \in \mathcal{B}} \left\{ \kappa(a_s) - a_s Z_s(1-N_s) - \theta a_s Z_s \textbf{1}_{\tau \leq s,  D=1} \right\}, 
\end{align*}
then 
$$ R_t^{a} = U_A\left(Y_t - \int_0^t \kappa(a_s) ds \right)$$
satisfies a Martingale Optimality Principle for the Agent's problem in this setting. 

\end{Lemma}

\begin{proof}
By construction, $R_T^{a} = U_A\left(W - \int_0^T \kappa(a_s) ds\right)$ and $R_0^a$ is independent of the Agent's action $a$. 
As in Section \ref{sec:OC}, we compute the variations of $R^{a}$ to obtain: 
\begin{align*}
dR^{a}_s &= -\gamma_A R_s^{a} Z_s((1-N_s) + N_s \textbf{1}_{D=1}) dB_s^{a} + R_s^{a}(e^{-\gamma_A K_s}-1)(1-N_s)dM_s \\
&+R_s^{a} \gamma_A \left( \frac{1}{2} \gamma_A Z_s^{2}((1-N_s) + N_s \textbf{1}_{D=1}) - f(s, Z_s, K_s) + \kappa(a_s) + (\lambda K_s + \frac{\lambda}{\gamma_A} (e^{-\gamma_A K_s}-1)) (1-N_s)\right)\\
&+ R_s^{a}\gamma_A\left( -a_s Z_s (1-N_s) + \theta a_s N_s \textbf{1}_{D=1}\right)\\
&= -\gamma_A R_s^{a} Z_s((1-N_s) + N_s \textbf{1}_{D=1}) dB_s^{a} + R_s^{a}(e^{-\gamma_A K_s}-1)(1-N_s)dM_s \\
&+R_s^{a} \gamma_A \left( \frac{1}{2} \gamma_A Z_s^{2}((1-N_s) + N_s \textbf{1}_{D=1}) - f(s, Z_s, K_s) + \kappa(a_s^{0}(1-N_s))+  \kappa(\theta a_s^{1} N_s \textbf{1}_{D=1}) + (\lambda K_s + \frac{\lambda}{\gamma_A} (e^{-\gamma_A K_s}-1)) (1-N_s)\right)\\
&+ R_s^{a}\gamma_A\left( -a_s^{0} Z_s (1-N_s) + \theta a_s^1 N_s \textbf{1}_{D=1}\right)\\
\end{align*}
and therefore $\mathbb{R}^{a}$ is a super-martingale for every $a$ in $\mathcal{B}$. Setting : 
$$ {a^0_s}^*(z) = -A\textbf{1}_{\frac{z}{\kappa} < -A} +  A\textbf{1}_{\frac{z}{\kappa} > A} + \frac{z}{\kappa}\textbf{1}_{ -A\leq \frac{z}{\kappa} \leq A}$$
and 
$$ {a^1_s}^*(z) = -A\textbf{1}_{ \frac{\theta z}{\kappa} \leq -A} +  A\textbf{1}_{\frac{\theta z}{\kappa} >  A} + \frac{\theta z}{\kappa}\textbf{1}_{-A\leq \frac{\theta z}{\kappa} \leq A},$$
then
$$ a^*_t = {a_t^{0}}^*\textbf{1}_{t < \tau} + {a_t^{1}}^* \textbf{1}_{t \geq \tau}.$$
We get that $\R^{a^*}$ is a  $\mathbb{P}^{a^*}$-martingale and the Agent's response given $W$ is then $a^*.$

\end{proof}

It remains to show that there actually exists a unique solution to 
$(Y, Z, K)$ in $\mathbb{S}^{2}_{\mathbb{G}} \times \mathbb{H}^2_{\mathbb{G}} \times  \mathbb{H}^2_{\mathbb{G}} $ to : 
$$ Y_t = W - \int_t^T Z_s((1-N_s) + N_s \textbf{1}_{D=1})dB_s - \int_t^T K_s(1-N_s) dM_s - \int_t^T f(s, Z_s, K_s) ds,$$
where 
\begin{align*}
 f(s, Z_s, K_s) &=  \left( \lambda K_s + \frac{\lambda}{\gamma_A} (e^{-\gamma_A K_s}-1)\right)(1-N_s) \\&+
 \frac{1}{2} \gamma_A Z_s^2 ((1-N_s) + N_s \textbf{1}_{D=1})
+ \inf_{a \in \mathcal{B}} \left\{ \kappa(a_s) - a_s Z_s(1-N_s) - \theta a_s Z_s \textbf{1}_{\tau \leq s,  D=1} \right\}.
\end{align*}

To to this, first note for any $s$ in $[0,T]$ fixed, and for any $t \in [s, T]$ there exists a unique pair $(Y^{i}, Z^{i}) \in \mathbb{S}^{2}_{\mathbb{G}}\times \mathbb{H}^{2}_{\mathbb{G}}$ solution to the BSDE : 
\begin{equation}
\label{Brown}
 Y^i_t(s)  = W_T^{1,1}(s)\textbf{1}_{D=1} - \int_t^T f^{1}(Z_s^{i}(s)) ds - \int_t^T Z_s^i(s) dB_s, 
 \end{equation}
where $f^{1}(z) = \frac{1}{2}\gamma_A z^2 + \inf_{a \in \mathcal{B}}(\kappa(a) - \theta a z)$ and where the notation $(Y^i(s), Z^i(s))$ is used to emphasize the dependency in $s$ of the terminal condition and its effect on the solution.  This existence result simply follows from the fact that for each $s$, (\ref{Brown}) is now simply a Brownian BSDE that fits into the classical quadratic setting of Briand and Hu. We may then set : 
$$ \tilde{W} = Y_\tau^i(\tau) \textbf{1}_{\tau \leq T} \textbf{1}_{D=1} + W^{1, 0}_{T \wedge \tau}(\tau) \textbf{1}_{\tau \leq T} \textbf{1}_{D=0} + W^{0} \textbf{1}_{T < \tau},$$
which is a $\mathbb{G}_{T \wedge \tau}$ measurable random-variable. We set : 
$$ f^{2}(z,k) = \frac{1}{2}\gamma_A z^2  +  \lambda k + \frac{\lambda}{\gamma_A} (e^{-\gamma_A k}-1) + \inf_{a \in \mathcal{B}}\left( \kappa(a)- aZ\right).$$

This fits right into the setting of the recent work \cite{Martin2020} on a default BSDE for Principal Agent problems. In particular, there exists a unique triplet $(\tilde{Y}, \tilde{Z}, \tilde{K} )$ in $ \mathbb{S}^{2}_{\mathbb{G}}\times \mathbb{H}^{2}_{\mathbb{G}} \times \mathbb{H}^{2}_{\mathbb{G}}$ such that~: 
$$ \tilde{Y}_t = \tilde{W} - \int_{t \wedge \tau}^{T \wedge \tau} \tilde{Z}_s dB_s -   \int_{t \wedge \tau}^{T \wedge \tau} \tilde{K}_s dM_s  - \int_{t \wedge \tau}^{T \wedge \tau} f^2(\tilde{Z}_s, \tilde{K}_s) ds. $$

Finally, setting : 
\begin{itemize}
\item $Y_t = \tilde{Y}_t (1-N_t) + Y^{i}_t(\tau) N_t \textbf{1}_{D=1}$
\item $Z_t = \tilde{Z}_t (1-N_t) + Z^{i}_t(\tau) N_t \textbf{1}_{D=1}$
\item $K_t = \tilde{K}_t (1-N_t)$
\end{itemize}
and noting that : 
$$ f(s, z,k) = f^{1}(z)N_s \textbf{1}_{D=1} + f^{2}(z,k)(1-N_s),$$
we obtain that $(Y,Z,K)$ is a solution to : 
\begin{align} 
\label{BS}Y_t = W - \int_t^T Z_s((1-N_s) + N_s \textbf{1}_{D=1})dB_s - \int_t^T K_s(1-N_s) dM_s - \int_t^T f(s, Z_s, K_s) ds.
\end{align}
Finally, uniqueness holds through a classical reasoning, noting that up to a change in probability, $Y$ is a local martingale that has continuous paths on a certain left-hand neighbourhood of $T$.

\end{document}